\documentclass[12pt]{amsart}
\usepackage{amsmath,amsfonts,amsthm,amscd,amssymb}
\usepackage{tipa}
\usepackage{stmaryrd}
\newtheorem{theorem}{Theorem}
\newcommand{\beq}{\begin{eqnarray}}

\newcommand{\eeq}{\end{eqnarray}}
\newcommand{\R}{\mathbb{R}}
\newcommand{\RR}{\mathcal{R}}
\newcommand{\tr}{\mathrm{tr}}

\newcommand{\tf}{\mathrm{tf}}

\newcommand{\stl}{\stackrel}
\newcommand{\slashd}{\ipaclap{\textipa{d}}{\textipa{/}}}%needs the "tipa" package

\newcommand{\bkill}{\Box_{\mathcal{K},g}}
\newtheorem{proposition}[theorem]{Proposition}
\newtheorem{lemma}[theorem]{Lemma}

\newtheorem{corollary}[theorem]{Corollary}
\newtheorem{definition}[theorem]{Definition}
\newtheorem{remark}[theorem]{Remark}
\numberwithin{equation}{section}
\numberwithin{theorem}{section}
\usepackage{amsmath}
\usepackage{amssymb}
\usepackage{latexsym}
\begin{document}
\bibliographystyle{amsalpha} 
\title[Asymptotics]{Asymptotics of the self-dual\\ [2pt] deformation complex}
\author{Antonio G. Ache}
\author{Jeff A. Viaclovsky}
\address{Department of Mathematics, University of Wisconsin, Madison, WI 53706}
\email{ache@math.wisc.edu, jeffv@math.wisc.edu}
\thanks{Research partially supported by NSF Grants DMS-0804042 and DMS-1105187}
\date{January 4, 2012}
\begin{abstract} 
We analyze the indicial roots of the self-dual 
deformation complex on a cylinder $(\R \times Y^3, dt^2 + g_Y)$, 
where $Y^3$ is a space of constant curvature. 
An application is the optimal decay rate of solutions on a
self-dual manifold with cylindrical ends having cross-section $Y^3$. 
We also resolve a conjecture of Kovalev-Singer in the 
case where $Y^3$ is a hyperbolic rational homology $3$-sphere, and show that there are
infinitely many examples for which the conjecture is 
true, and infinitely many examples for which the conjecture is false.
Applications to gluing theorems are also discussed. 
\end{abstract}
\maketitle
\setcounter{tocdepth}{1}
\tableofcontents

%%%%%%%%%%%%%%%%%%%%%%%%%%%%%%%%%%%%%%%%%%%%%%%%
\section{Introduction}
%%%%%%%%%%%%%%%%%%%%%%%%%%%%%%%%%%%%%%%%%%%%%

Let $(M^{4},g)$ be a four-dimensional Riemannian manifold and let 
$Rm$ denote the Riemannian curvature tensor of $g$. 
Recall that $Rm$ admits an orthogonal decomposition of the form
\begin{align}
\label{rm4}
Rm=W+\frac{1}{2}E\varowedge g+\frac{1}{24}R g\varowedge g,
\end{align}
where $W$ is the Weyl tensor, $E$ is the traceless Ricci tensor of $g$, 
$R$ is the scalar curvature,  and $\varowedge$ is the Kullkarni-Nomizu product.
If $(M^4, g)$ is oriented, there is a further
decomposition of \eqref{rm4}. 
The Hodge-$*$ operator associated to $g$ acting on $2$-forms is a mapping
 $*:\Lambda^{2} \mapsto\Lambda^{2}$ satisfying $*^2 = Id$, 
and $\Lambda^{2}$ admits a decomposition of the form 
\begin{align}\label{2formsplusminus}
\Lambda^{2}=\Lambda^{2}_{+} \oplus\Lambda^{2}_{-},
\end{align}
where $\Lambda^{2}_{\pm}$ are the $\pm 1$ eigenspaces of 
$\displaystyle{*|}_{\Lambda^{2}}$. Sections of 
$\Lambda^2_+$ and $\Lambda^2_-$ are called self-dual and 
anti-self-dual $2$-forms, respectively. 
  The curvature tensor can be viewed as an operator 
$\mathcal{R}: \Lambda^2 \rightarrow \Lambda^2$, and we let 
$\mathcal{W}$ and $\mathcal{E}$ denote the operators associated to 
the Weyl and traceless Ricci tensors, respectively. With respect to the 
decomposition \eqref{2formsplusminus}, the full curvature operator decomposes as
\begin{align}
\label{block}
{\mathcal R}=
\left(
\mbox{
\begin{tabular}{c|c}
&\\
$\mathcal{W}^+ + \frac{R}{12} I $&$ \frac{1}{2} \mathcal{E} \pi_- $\\ &\\
\cline{1-2}&\\
$\frac{1}{2} \mathcal{E} \pi_+$ & $\mathcal{W}^- + \frac{R}{12} I$\\&\\
\end{tabular}
} \right),
\end{align}
where $\pi_{\pm}$ is the projection onto 
$\Lambda^2_{\pm}$, and the 
self-dual and anti-self-dual Weyl tensors are defined by 
$\mathcal{W}^{\pm} = \pi_{\pm} \mathcal{W} \pi_{\pm}$. 

\begin{definition}{\em
 Let $(M^4,g)$ be an oriented four-manifold. 
Then $g$ is called {\em{self-dual}} 
if $\mathcal{W}^- = 0$, and $g$ is called {\em{anti-self-dual}} 
if $\mathcal{W}^+ = 0$. In either case $g$ is said to be 
{\em{half-conformally-flat}}. }
\end{definition}
By reversing orientation, a self-dual metric becomes an 
anti-self-dual metric, so without loss of generality, we 
will only consider self-dual metrics. 

 Since Poon's example of a $1$-parameter family of self dual 
metrics on $\mathbb{CP}^2 \# \mathbb{CP}^2$ in 1988 \cite{Poon}, 
there has been an explosion of examples of self-dual 
metrics on various four-manifolds. We do not attempt to 
give a complete history here, but only mention a few 
results closely related to our results in this paper.  In 1989, 
Donaldson and Friedman developed a twistor space gluing procedure 
which invoked many non-trivial results in algebraic geometry 
\cite{DonaldsonFriedman}.
In 1991, Floer produced examples on $n \#  \mathbb{CP}^2$
by an analytic gluing procedure \cite{Floer}. 
Then in 2001, Kovalev and Singer 
generalized Floer's analytic gluing result to cover the case of gluing 
orbifold self-dual metrics~\cite{KovalevSinger}. 
We will describe the relation of our 
work with these prior works in more detail below,
but first will state our main results. 

Since the SD condition is conformally invariant, we are 
free to conformally change an end to obtain different 
types of asymptotics. For simplifying computations, the most useful 
type of geometry is that of cylindrical ends:
\begin{definition}
\label{ALCdef}
{\em
Let $(Y^{3},g_Y)$ be a compact $3$-manifold with constant curvature.
A complete Riemannian manifold $(M^4,g)$ 
is called {\em{asymptotically cylindrical}} 
or {\em{AC}} with cross-section $Y$ of order $\tau$ if 
there exists a diffeomorphism 
$\psi : M \setminus K \rightarrow  \R_+ \times Y$ 
where $K$ is a subset of $M$ containing all other ends, 
satisfying
\begin{align}
\label{eqgdfac1in}
(\psi_* g)_{ij} &= (g_C)_{ij} + O( e^{-\tau t}),\\
\label{eqgdfac2in}
\ \partial^{|k|} (\psi_*g)_{ij} &= O( e^{-\tau t}),
\end{align}
for any partial derivative of order $k$, as
$t \rightarrow \infty$, 
where $g_C = dt^2 + g_{Y}$ is the product cylindrical metric. 
}  
\end{definition}

 Self-dual metrics have a rich obstruction theory. If 
$(M,g)$ is a self-dual four-manifold, the deformation complex is given by 
\begin{align}
\Gamma(T^*M) \overset{\mathcal{K}_g}{\longrightarrow} 
\Gamma(S^2_0(T^*M))  \overset{\mathcal{D}}{\longrightarrow}
\Gamma(S^2_0(\Lambda^2_-)),
\end{align}
where $\mathcal{K}_g$ is the conformal Killing operator defined 
by 
\begin{align}
( \mathcal{K}_g(\omega))_{ij} = \nabla_i \omega_j + \nabla_j \omega_i - 
\frac{1}{2} (\delta \omega) g, 
\end{align}
with $\delta \omega = \nabla^i \omega_i$, 
and $\mathcal{D} = (\mathcal{W}^-)_g'$ is the linearized anti-self-dual Weyl curvature 
operator. This complex is then wrapped-up into a single 
operator 
\begin{align}
F :\Gamma(S^2_0(T^*M)) \longrightarrow \Gamma(S^2_0(\Lambda^2_-))
\oplus \Gamma(S^2_0(T^*M)),
\end{align}
defined by 
\begin{align}
F (h) = ( \mathcal{D} h, 2 \delta h),
\end{align}
and $(\delta h)_j = \nabla^i h_{ij}$. 
This operator is mixed order elliptic of order $(2,(0,1))$  
in the sense of Douglis-Nirenberg \cite{DN}, with 
formal $L^2$-adjoint 
 \begin{align}
F^* :\Gamma(S^2_0(\Lambda^2_-))
\oplus \Gamma(S^2_0(T^*M)) \longrightarrow\Gamma(S^2_0(T^*M)),
\end{align}
given by 
\begin{align}
F^* ( Z, \omega) = \mathcal{D}^* Z - \mathcal{K}_g \omega. 
\end{align}
\begin{definition}{\em
The {\em{indicial roots}} of $F$ are those complex numbers $\lambda$ for which 
there is a solution $h$ of $F(h) = 0$ such that the components of $h$ 
have the form $e^{\lambda t} p(y,t)$ where $p$ is a polynomial in $t$ with 
coefficients in $C^{\infty}(Y)$. The indicial roots of $F^*$ are 
defined analogously for pairs $(Z, \omega)$.}
\footnote{Our definition of indicial roots differs from that in \cite{LockhartMcOwen} 
by a factor of $\sqrt{-1}$.}
\end{definition}

We will first determine the indicial roots of $F^*$. The indicial 
roots of $F$ can then be obtained by using an index theorem, 
as we will show below. One could equivalently first analyze
the indicial roots of $F$, however, for purposes of computation 
it turns out to be somewhat easier to completely analyze the cokernel (although
the computations are in principle equivalent). 

\subsection{Spherical cross-section}

Our first result deals with cross-section $Y$ having 
constant positive curvature. In Theorem \ref{s3thm2} below, we 
determine {\em{all}} indicial 
roots of $F^*$, but for simplicity we only state the following here in the introduction:
\begin{theorem}
\label{s3thm}
Let $M$ be  $\R \times S^{3}/ \Gamma$ with 
product metric $g = dt^2 + g_{S^3/\Gamma}$, where $g_{S^3/\Gamma}$ is a 
metric of constant curvature $1$.  
Let $\mathcal{I}^*$ denote the set of indicial roots  of $F^*$. 
If $\beta \in \mathcal{I}^*$ satisfies $|Re(\beta)| < 2$ 
then $\beta = 0$ or $\beta = \pm 1$. 
In these cases, the corresponding solutions are of the form $(0,\omega)$,
where $\omega$ is dual to a conformal Killing field 
(that is, $\mathcal{K}_g \omega = 0$). Consequently, 

\begin{itemize}
\item Case (0): $0 \in \mathcal{I}^*$, and the corresponding solutions are given by $(0,dt)$, or
$(0,\omega_0)$ for $\omega_0$ dual to a Killing field on $S^3/\Gamma$.

\item Case (1): $\pm 1 \in \mathcal{I}^*$ if and only if $\Gamma$ is trivial. 
In this case, the corresponding solutions are given by $(0,\omega)$, 
where $\omega$ is given by $e^{\pm t} ( \phi dt \mp d\phi)$ where $\phi$ 
is a lowest nontrivial eigenfunction of $\Delta_{S^3}$ with eigenvalue $3$.
\end{itemize}
\end{theorem}
\begin{remark}
\label{rmk1}
{\em The indicial roots $\beta \in \mathcal{I}^*$ satisfying 
$|Re(\beta)| \geq 2$ fall into two classes. The indicial roots in the 
first class are integers and the corresponding solutions are of the form $(Z,0)$; 
these are Cases (2) and (3) in Theorem \ref{s3thm2}. The indicial roots in 
the other class have non-zero imaginary part, and the corresponding solutions are 
of the form $(Z,\omega)$ with $Z$ nontrivial and $\mathcal{K}_g(\omega) \neq 0$;
these are Cases (4) and (5) in Theorem~\ref{s3thm2}.
}
\end{remark}
We can also completely characterize the indicial roots of the forward 
operator $F$. 
This follows from the above determination of the 
cokernel indicial roots, together with the index theorem of Lockhart and McOwen; 
it turns out that these are the same.
We will describe all kernel elements explicitly below in Theorem \ref{kerlong}, 
but for purposes of brevity in the introduction we 
only state here the following theorem 
which generalizes a well-known result of Floer \cite{Floer}. 
In order to state the theorem, we define the symmetric 
product of $1$-forms $\omega_1$ and $\omega_2$ by 
\begin{align}
\omega_1 \odot \omega_2 = \omega_1 \otimes \omega_2+ \omega_2 \otimes \omega_1. 
\end{align}
\begin{theorem}
\label{kershort}
 Let $M$ be  $\R \times S^{3}/ \Gamma$ with product metric 
$g = dt^2 + g_{S^3/ \Gamma}$, where $g_{S^3/ \Gamma}$ is a metric of
constant curvature $1$, 
and let $\mathcal{I}$ denote the set of indicial roots  of $F$.
Then $\mathcal{I} = \mathcal{I}^*$. 
For $\beta = 0 \in \mathcal{I}$,
the corresponding solutions of $F(h) = 0$ are given by 
\begin{align}
\label{ker0}
span \{  3 dt \otimes dt - g_{S^3},dt \odot \omega_0  \},
\end{align}
where $\omega_0$ is a dual to a Killing field on $S^3/ \Gamma$. 

Next, $\beta = \pm 1 \in \mathcal{I}$ if and only if $\Gamma$ is trivial. 
In this case, the corresponding kernel elements are given by 
\begin{align}
\label{ker1}
h_{\phi} = p(t) \phi ( 3 dt \otimes dt - g_{S^3}) + q(t) (dt \odot d\phi), 
\end{align}
where $p(t) = C_3 e^t - C_4 e^{-t}$ and $q(t) = C_3 e^t + C_4 e^{-t}$, 
for some constants $C_3$ and $C_4$, and $\phi$ is a lowest nonconstant eigenfunction of 
$\Delta_{S^3}$. 

 Morever, solutions in \eqref{ker0} and \eqref{ker1} are in the 
image of the conformal Killing operator. 
All other indicial roots $\beta \in \mathcal{I}$ satisfy 
$|Re(\beta)| \geq 2$. 
\end{theorem}

\begin{remark}{\em As in Remark \ref{rmk1}, the indicial roots $\beta \in \mathcal{I}$ satisfying 
$|Re(\beta)| \geq 2$ fall into two classes. The indicial roots in the 
first class are integers and the 
corresponding solutions are not in the image of the conformal Killing operator;
these are Cases (2) and (3) in Theorem \ref{kerlong}. The indicial 
roots in the other class have non-zero imaginary part, and the corresponding 
solutions {\em{are}} in the image of the conformal Killing operator;
these are Cases (4) and (5) in Theorem \ref{kerlong}.
}
\end{remark}
A corollary is the optimal result:
\begin{corollary}
\label{t1s3}
Let $(M^4,g)$ be the the cylinder $\R \times Y^3$, 
where $Y^3 = S^3/\Gamma$ with
$\Gamma \subset SO(4)$ a finite subgroup acting freely 
on $S^3$ with product metric $g = dt^2 + g_Y$, where $g_Y$ is a 
metric of constant curvature $1$. 
\begin{enumerate}
\item[(a)]
Let $(Z, \omega)$ be a solution of $\mathcal{D}^*Z = \mathcal{K}_g \omega$. 
If $Z = o(e^{2|t|})$ and $\omega = o(e^{2|t|})$ as  
$|t| \rightarrow \infty$ then  $Z = 0$, and 
$\omega$ is 
dual to a conformal Killing field.
\item[(b)]
 Let $h$ be a solution of $\mathcal{D} h = 0$ and $\delta h =0$. 
If $h = o(e^{2|t|})$ as  $|t| \rightarrow \infty$ then
$h$ can be written as a linear combination of elements in \eqref{ker0} and \eqref{ker1}. 
\end{enumerate}
\end{corollary}

Standard analysis in weighted spaces then 
implies the following corollary for AC manifolds with spherical 
cross-section:
\begin{corollary}
\label{t1ac}
 Let $(M^4,g)$ be self-dual and asymptotically cylindrical
with cross-section $(Y^3 = S^3/\Gamma, g_Y)$ with  
$\Gamma \subset SO(4)$ a finite subgroup acting freely 
on $S^3$, where $g_Y$ is a metric of constant curvature $1$. 
\begin{enumerate}
\item[(a)]
Let $(Z, \omega)$ be a solution of $\mathcal{D}^*Z = \mathcal{K}_g \omega$. 
If $Z = o(e^{2t})$ and $\omega = o(e^{2t})$ then $\omega$ is 
dual to a conformal Killing field, 
and $Z = O( e^{-2t})$ as $t \rightarrow \infty$. 
\item[(b)]
 Let $h$ be a solution of $\mathcal{D} h = 0$ and $\delta h =0$. 
If $h = o(e^{2|t|})$ as  $|t| \rightarrow \infty$ then
$h$ has an asymptotic expansion with leading term 
as in \eqref{ker0} or \eqref{ker1}. 
\end{enumerate}
\end{corollary}
In Section 8, we apply Corollary \ref{t1ac} 
to fix a gap in the proof of a key step in
 the main gluing result in \cite{KovalevSinger}. 

To state the next result, we require the following definition. 
\begin{definition}
\label{ALEdef}
{\em
 A complete Riemannian manifold $(M^4,g)$ 
is called {\em{asymptotically locally 
Euclidean}} or {\em{ALE}} of order $\tau$ if it has finitely 
many ends, and for each end there exists a finite subgroup 
$\Gamma \subset SO(4)$ 
acting freely on $S^3$ and a 
diffeomorphism 
$\psi : M \setminus K \rightarrow ( \mathbf{R}^4 \setminus B(0,R)) / \Gamma$ 
where $K$ is a subset of $M$ containing all other ends, 
and such that under this identification, 
\begin{align}
\label{eqgdfale1in}
(\psi_* g)_{ij} &= \delta_{ij} + O( r^{-\tau}),\\
\label{eqgdfale2in}
\ \partial^{|k|} (\psi_*g)_{ij} &= O(r^{-\tau - k }),
\end{align}
for any partial derivative of order $k$, as
$r \rightarrow \infty$, where $r$ is the distance to some fixed basepoint.  
}
\end{definition}
It is known that any self-dual ALE metric is ALE of order $2$,
after a possible change of coordinates at infinity. ALE of 
any order $\tau < 2$ was first shown by \cite{TV}, 
while ALE of order exactly $2$ was shown in \cite{Streets},
see also \cite{Chen, AcheViaclovsky}. This order is 
optimal, so without loss of generality we will assume 
that all ALE spaces are ALE of order $2$. 

We also have the following optimal decay result for self-dual ALE spaces:
\begin{theorem}
\label{t1a}
Let $(M,g)$ be self-dual and asymptotically locally Euclidean. 
\begin{enumerate}
\item[(a)]
Any solution of $\mathcal{D}^* Z = \mathcal{K}_g \omega$ satisfying 
$Z = o(1)$ and $\omega = o(r^{-1})$ 
must satisfy $\omega = 0$ and $Z = O( r^{-4})$ as $r \rightarrow \infty$.
\item[(b)]
Any solution of $\mathcal{D} h = 0$ and $\delta h =0$
satisfying $h = o(1)$ must satisfy $Z = O( r^{-2})$ as $r \rightarrow \infty$.
\end{enumerate}
\end{theorem}

\subsection{Hyperbolic cross-section}
We first define
\begin{align}
H^1_{C} (Y) = \{ B \in S^2_0(T^*Y) \ | \ d^{\nabla} B = 0, tr(B) = 0\},
\end{align}
where $(d^{\nabla}B)_{klj}$ is given by 
$(d^{\nabla}B)_{klj}=\nabla_{k}B_{lj}-\nabla_{l}B_{kj}$,
to be the vector space of traceless Codazzi tensor fields. 
For the case of hyperbolic cross-section, we have the following:
\begin{theorem}
\label{hypthm}
Let $(M^4,g)$ be the the cylinder $\R \times Y^3$, 
where $(Y^3,g_Y)$ is compact and hyperbolic with constant curvature $-1$,
with product metric $g = dt^2 + g_Y$, and let $\mathcal{I}^*$
denote the set of indicial roots of $F^*$. Then there exists an 
$\epsilon > 0$ such that if $\beta \in \mathcal{I}^*$ with 
$| Re(\beta)| < \epsilon$ then $\beta = 0$. 
The corresponding kernel of $F^*$ has dimension 
\begin{align}
\label{hypdc}
1 + b_1(Y) + 2 \dim (H^1_C(Y)). 
\end{align}
The corresponding kernel of $F$ has the same dimension
and is spanned by 
\begin{align}
\{ 3 dt \otimes dt - g_Y,  dt \odot \omega, cos(t)\cdot B, \sin(t) \cdot B\},
\end{align}
where $\omega$ is any harmonic $1$-form $\omega$,
and $B$ is any traceless Codazzi tensor on $Y^3$.
\end{theorem}

\begin{remark}{\em The element $(0,dt)$ is in the cokernel, 
which accounts for the $1$ in \eqref{hypdc}.
The other cokernel elements in case $b_1(Y) \neq 0$ 
arise from non-trivial harmonic $1$-forms, 
and those in case
$H^1_C(Y) \neq \{0\}$ of course arise from non-trivial traceless Codazzi tensor 
fields. 
These elements are written down explicitly in 
Section \ref{adjointeq}, see Propositions \ref{adjtype3}(b)
and \ref{sec5prop}(b). We only note here that the nontrivial 
solutions in this case satisfy $Z = O(1)$ as $|t| \rightarrow \infty$ or are 
periodic in $t$.
}
\end{remark}
We define\footnote{Note that this definition is more general than the definition 
in \cite[Section 4.2.1]{KovalevSinger}
in that we allow solutions which have polynomial growth in $t$.} 
\begin{align*}
H^2_+( \R \times Y^3) = \{ &Z \in S^2_0(\Lambda^2_-) \ | \ \mathcal{D}^* Z = 0
\mbox{ and } Z = O(e^{\epsilon |t|}) \\
& \mbox{ as } |t| \rightarrow \infty \mbox{ for every } \epsilon > 0 \}.
\end{align*}
In \cite[Conjecture 4.11]{KovalevSinger}, 
it was conjectured that $H^2_+( \R \times Y^3)  = \{0\}$ for any hyperbolic 
rational homology $3$-sphere. 
Theorem \ref{hypthm} shows that this is true if and only if $Y^3$ does not admit 
any non-trivial traceless Codazzi tensor field.
Using this, and some examples of certain hyperbolic $3$-manifolds 
of \cite{Kapovich,DeBlois}, we obtain infinitely many examples for
which the conjecture is 
true, and infinitely many examples for which the conjecture is false:
\begin{theorem} 
\label{hrex} Let $(Y^3,g_Y)$ be a hyperbolic rational homology $3$-sphere,
with $g_Y$ of constant curvature $-1$, 
and $M = \R \times Y^3$ with the product metric $g = dt^2 + g_Y$. 
Then $H^2_+(\R \times Y^3) = \{0\}$ if and only if $Y^3$ admits no 
non-trivial traceless Codazzi tensor fields. 
Furthermore, there are infintely many hyperbolic rational homology $3$-spheres
satisfying $H^2_+(\R \times Y^3)= \{0\}$, and infinitely many 
satisfying  $H^2_+ ( \R \times Y^3) \neq \{0\}$.
\end{theorem}

We also have the following application to AC manifolds with 
hyperbolic cross-section:
\begin{corollary}
\label{hypthm2}
 Let $(M^4,g)$ be self-dual and asymptotically cylindrical with 
cross-section $(Y^3,g_Y)$ a hyperbolic rational homology $3$-sphere 
with $g_Y$ of constant curvature $-1$, satisfying $H^1_C(Y) = \{ 0 \}$. 
\begin{enumerate}
\item[(a)]
Let $(Z, \omega)$ be a solution of $\mathcal{D}^*Z = \mathcal{K}_g \omega$. 
Then there exists a constant $\epsilon > 0$, such that if
$(Z, \omega)$ solves  $\mathcal{D}^*Z = \mathcal{K}_g \omega$
and satisfies $Z = o(e^{\epsilon |t|})$ and $\omega = o(e^{\epsilon |t|})$ as  
$t \rightarrow \infty$ 
then  $\omega$ is dual to a conformal Killing field and
$Z = o(e^{-\epsilon |t|})$ as $t \rightarrow \infty$. 

\item[(b)] Let $h$ be a solution of 
$\mathcal{D} h = 0$ and $\delta h = 0$. 
Then there exists a constant $\epsilon > 0$, such that if
$h = o(e^{\epsilon |t|})$ as $|t| \rightarrow \infty$, then 
$h$ admits an expansion 
\begin{align} 
h = c \cdot( dt \otimes dt - 3 g_Y) + O(e^{-\epsilon |t|})
\end{align}
for some constant $c$ as  $|t| \rightarrow \infty$.
\end{enumerate}
\end{corollary}

%%%%%%%%%%%%%%%%%%
\subsection{Flat cross-section}
Finally, in the case that $(Y^3, g)$ is a flat torus, we have the following:
\begin{theorem}
\label{flatthm}
Let $(M^4,g)$ be the the cylinder $\R \times Y^3$, 
where $(Y^3,g_Y)$ is compact and flat, with product metric 
$g = dt^2 + g_{Y}$, and let $\mathcal{I}$
denote the set of indicial roots of $F$. Then there exists an 
$\epsilon > 0$ such that if $\beta \in \mathcal{I}$ with 
$| Re(\beta)| < \epsilon$ then $\beta = 0$. 
The corresponding kernel of $F$ has dimension $14$ 
and is spanned by 
\begin{align}
\{ 3 dt \otimes dt - g_Y, dt \odot \omega, B, tB \},
\end{align}
where $\omega$ is any parallel $1$-form and $B$ 
is any parallel traceless symmetric $2$-tensor on~$Y^3$.

The corresponding cokernel of $F$ has dimension $14$ and is spanned by 
\begin{align}
\{ (0, dt), (0, \omega_0), (Z, 0), (tZ,0) \}, 
\end{align}
where $\omega_0$ is a parallel $1$-form, and $Z$ is any parallel 
section of $S^2_0(\Lambda^2_-)$. 
\end{theorem}

\begin{remark}{\em One can easily use our computations to 
explicitly determine {\em{all}} indicial roots in the case of flat 
cross-section $Y^3 = T^3$. However, in the interest of 
brevity this is omitted. }
\end{remark}
\subsection{Remarks and outline of the paper}
We next give a brief outline of the paper. 
Sections \ref{asdsec} and \ref{linweylcyl} will be concerned with 
the derivation of the linearized anti-self-dual Weyl tensor 
in separated variables. In these sections, there is  
overlap with computations in Floer's paper \cite{Floer}. 
However, the main formula given in Floer for $(\mathcal{W}^-)'$ at a
cylindrical metric is incorrect \cite[Proposition 5.1]{Floer}
(in addition to mistakes in the coefficients, 
Floer's formula omits crucial terms involving the trace component 
$h_{00}$).
The correct formula (which moreover holds for any 
cross-section $Y^3$ with constant curvature) is given in Theorem \ref{mainprop}. 
Section~\ref{slashdsec} contains required formulas for a 
Dirac-type operator, as well as some necessary eigenvalue computations. 
Section~\ref{adjointeq} contains the core analysis of the 
kernel of $\mathcal{D}^*$.  The analysis 
in Section~\ref{mixedsec} is necessary to determine the possibilities 
for the $1$-form $\omega$ appearing in the adjoint equation. 
The proofs of all the main theorems are then completed in Section 
~\ref{proofsec}. 
In Section~\ref{gluesec}, we discuss the application of our results to 
gluing theorems. 

Finally, the Appendix contains the derivation 
of a crucial formula relating the square of the Dirac 
operator to the linearized Einstein equation on the 
cross-section. In the case of spherical cross-section, 
Floer writes down such a formula \cite[Lemma 5.1]{Floer}, 
but which has errors in the coefficients. 
The correct formula (which moreover holds for general 
cross-section $Y$) is given in Corollary \ref{Floercor}. 

\subsection{Acknowledgements}
The authors would like to thank Claude LeBrun for several 
discussions about the paper \cite{LeBrunMaskit}, and the 
relation with the gluing theorems given in \cite{KovalevSinger}. 
We would also like to thank Richard Kent for crucial help with the 
hyperbolic examples in Theorem \ref{hrex}.

%%%%%%%%%%%%%%%%%%%%%%%%%%%%%%%%%%%%%%%%%%%%%%%%%%%%
\section{The anti self-dual part of the curvature tensor}
\label{asdsec}
Let $(M^4, g)$ be an orientable $4$-manifold. As mentioned in the introduction, 
according to the decomposition 
\begin{align}
Rm= W^+ + W^- + \frac{1}{2}E\varowedge g+\frac{1}{24}R_{g}g\varowedge g,
\end{align}
we have the associated curvature operators are written as
\begin{align}\label{curvatureoperatordecomposition}
\RR=\mathcal{W}^{+}+\mathcal{W}^{-}+\mathcal{E}+\frac{1}{24}\mathcal{S} ,
\end{align}
where
\begin{align}\label{preservetype}
\left(\mathcal{W}^{\pm}+\frac{1}{24}\mathcal{S}\right):\Lambda^{2}_{\pm}\left(T^{*}M\right)\mapsto\Lambda^{2}_{\pm}\left(T^{*}M\right),
\end{align}
and 
\begin{align}\label{reversetype}
\mathcal{E}:\Lambda^{2}_{\pm}\left(T^{*}M\right)\mapsto\Lambda^{2}_{\mp}\left(T^{*}M\right).
\end{align}

Some basic properties of the tensors $W^{\pm}$ and of the curvature operators $\mathcal{W}^{\pm}$  are the following
\begin{enumerate}
\item Viewed as a $(1,3)$ tensor, for any $C^{2}$ function $f$,  $W^{\pm}(e^{-2f}g)=W^{\pm}(g)$.
\item Letting $\mathcal{C}:S^{2}\left(\Lambda^{2}\left(T^{*}M\right)\right)\mapsto S^{2}\left(T^{*}M\right)$ be the \emph{Ricci Contraction Map} defined by
$\left(\mathcal{C}U\right)_{ab}=g^{cd}U_{acbd}$, 
then 
\begin{align}\label{zeroriccicontraction}
\mathcal{C}W^{\pm}=0.
\end{align}
\item Both $\mathcal{W}^{+}$ and $\mathcal{W}^{-}$ are traceless. 
\end{enumerate} 

We note that our convention is that if $P_{ijkl}$ is a tensor satisfying 
$P_{ijkl} = -P_{jikl} = -P_{ijlk} = P_{klij}$, then the associated operator 
$\mathcal{P} : \Lambda^2 \rightarrow \Lambda^2$ is given by 
\begin{align}
\left(\mathcal{P}\omega\right)_{ij}=\frac{1}{2}\sum_{k,l}P_{ijkl}\omega_{kl}.
\end{align}

%%%%%%%%%%%%%%%%%%%%%%%%%%%%%%%%%%%%
\subsection{The anti self-dual part of the Weyl tensor as a bilinear form}\label{wey}

Consider a warped product metric on $M = \mathbb{R} \times Y^3$ of the form 
\begin{align}\label{eqwey1}
g=dt^2+g_{Y},
\end{align}
Where $g_{Y}$ is a smooth metric on $Y$, possibly depending on $t$. 
Our ultimate goal is a formula for the 
linearized anti-self-dual Weyl curvature $\mathcal{D}$, which maps from 
\begin{align}
\mathcal{D} : S^2_0 (T^*M) \rightarrow S^2_0 ( \Lambda^2_-). 
\end{align}
Using the decomposition 
$T^* M = \{dt\} \oplus T^*Y$, we have
\begin{align}\label{s2r4}
S^2 (T^* M) = S^2 ( dt) \oplus ( dt \odot T^*Y) \oplus S^2 (T^* Y), 
\end{align}
which we will write as 
\begin{align}
\tilde{h} = h_{00} dt\otimes dt + ( \alpha \otimes dt + dt \otimes \alpha) 
+ h.
\end{align}
Next, we have 
\begin{align}
\Lambda^2 ( dt \oplus T^* Y ) = (\Lambda^1(dt) \otimes \Lambda^1 (T^*Y)) 
 \oplus \Lambda^2( T^* Y).
\end{align}
Given an orientation, we then have 
\begin{align}
\Lambda^2 ( dt \oplus T^* Y) = (\Lambda^1(dt) \otimes \Lambda^1 (T^*Y)) 
 \oplus \Lambda^1( T^* Y).
\end{align}
Under this decomposition, the self-dual forms correspond to 
\begin{align} 
dt \wedge \alpha + \tilde{*} \alpha, 
\end{align}
while the anti-self-dual forms correspond to 
\begin{align}
dt \wedge \alpha - \tilde{*} \alpha, 
\end{align}
where $\tilde{*}$ is the Hodge-$*$ operator on $Y$. 
Consequently, we have the identification
\begin{align}
S^2_0 ( \Lambda^2_-) = S^2_0 ( T^* Y),
\end{align}
and we can therefore view $\mathcal{D}$ as a mapping 
\begin{align}
\mathcal{D} : S^2_0(T^*M) \rightarrow  S^2_0 ( T^* Y).
\end{align}
In order to proceed, we must first 
write down $\mathcal{W}^-$, considered as an element of $S^2_0 ( T^* Y)$.
\begin{proposition}
\begin{align}\label{eqwey2}
(\mathcal{W}^{-})_{ij}=\Phi_{ij}-\Psi_{ij}+\Omega_{ij},
\end{align}
where 
\begin{align}
\Phi_{ij}
&=\tf(R_{0i0j}(g)),\label{phi}\\
\Psi_{ij}
&=\tf\left(\mathrm{Sym}_{ij}\left(\sum_{k,l}\epsilon_{ikl}R_{0jkl}(g)\right)\right),\label{psi}\\
\Omega_{ij}
&=\frac{1}{4}\tf\left(\sum_{k,l}\sum_{p,q}\epsilon_{ikl}\epsilon_{jpq}R_{klpq}(g)\right)\label{Omega},
\end{align}
where the symbols $\epsilon_{ijk}$ are the components of the 
volume element defined by 
\begin{align}\label{epsilon}
e_{i}\wedge e_{j}\wedge e_{k}=\epsilon_{ijk}e_{1}\wedge e_{2}\wedge e_{3},
\end{align}
$\mathrm{Sym}$ in \eqref{psi} denotes the \emph{Symmetrization Operator} given by 
\begin{align}
\mathrm{Sym}_{ij}(F)=\frac{1}{2}\left(F_{ij}+F_{ji}\right),
\end{align}
and for $h\in S^{2}\left(T^{*}Y\right)$,  $\tf(h)$ 
denotes the traceless component of $h$ with respect to the metric $g_Y$.
\end{proposition}
\begin{proof}
Given $g$ as in \eqref{eqwey1} and considering the decomposition \eqref{curvatureoperatordecomposition}, we conclude from \eqref{curvatureoperatordecomposition},\eqref{preservetype} and \eqref{reversetype} that for any $\omega,\omega^{\prime}\in\Lambda^{2}_{-}\left(M\right)$
\begin{align*}
\left\langle \left(\mathcal{W}^{-}+\frac{1}{24}\mathcal{S}\right)\omega,\omega^{\prime}\right\rangle=\langle\RR\omega,\omega^{\prime}\rangle.
\end{align*}
In order to compute $\langle\RR\omega,\omega^{\prime}\rangle$ we use the isomorphism between $\Omega^{1}\left(Y\right)\oplus\Omega^{1}\left(Y\right)$ and $\Lambda^{2}\left(\R\times Y\right)$ given as follows: any $2$-form $\omega$ in $\Lambda^{2}\left(\R\times Y\right)$ can be written uniquely as 
\begin{align}\label{isomorphism}
\omega=dt\wedge\pi^{*}\left(\xi\right)+\pi^{*}\left(\tilde{*}\eta\right), 
\end{align}
where $\pi$ is the projection map $\pi:\R\times Y\mapsto Y$, $\xi$ and $\eta$ are $1$-forms in $\Omega^{1}(Y)$ and $\tilde{*}$ is the Hodge-$*$ operator with respect to the metric $g_Y$. Given a local orthonormal oriented basis $\{e_{1},e_{2},e_{3}\}$ of $\Gamma\left(TY\right)$, the operator $\tilde{*}:\Omega^{1}(Y)\mapsto\Omega^{2}(Y)$ takes the form
\begin{align}\label{hodgeS3}
\tilde{*}\left(\zeta\right)_{ij}=\sum_{k=1}^{3}\epsilon_{ijk}\zeta_{k}.
\end{align}
 Using $i,k,l$ to denote indices in $\{1,2,3\}$ we see that the $1$-forms $\xi$, $\eta$ in \eqref{isomorphism} can be written in coordinates as
\begin{align*}
\xi_{k}=\omega_{0k},~
\eta_{i}=\frac{1}{2}\sum_{j,k}\epsilon_{ijk}\omega_{jk}=
\sum_{j<k}\epsilon_{ijk}\omega_{jk},
\end{align*} 
Moreover, in these coordinates, the Hodge-$*$ operator can be computed as
\begin{align}\label{hodgeomega}
*\omega=dt\wedge\pi^{*}\left(\eta\right)+\pi^{*}\left(\tilde{*}\xi\right),
\end{align} 
therefore,  $\omega\in\Lambda^{2}_{\pm}\left(T^{*}Y\right)$ if and only if $\xi=\pm\eta$. Given $2$-forms $\omega,\omega^{\prime}\in\Lambda^{2}\left(M\right)$, let us write  $\omega=dt\wedge\pi^{*}(\xi)+\pi^{*}(\tilde{*}\eta)$ and $\omega^{\prime}=dt\wedge\xi^{\prime}+\pi^{*}(\tilde{*}\eta^{\prime})$ , then we can express $\langle\RR\omega,\omega^{\prime}\rangle$ as
\begin{align}\label{generalcurvatureoperator}
\langle \RR \omega,\omega^{\prime}\rangle=\alpha_{ij}\xi_{i}\xi^{\prime}_{j}+\beta_{ij}\xi_{i}\eta^{\prime}_{j}+\beta_{ji}\eta_{i}\xi^{\prime}_{j}+\gamma_{ij}\eta_{i}\eta^{\prime}_{j}
\end{align}
where clearly
\begin{align}\label{curvatureoperatorcomponents}
\alpha_{ij}=R_{0i0j}(g),~ \beta_{ij}=\frac{1}{2}\sum_{k,l}\epsilon_{jkl}R_{0ikl}(g),~\gamma_{ij}=\frac{1}{4}\sum_{k,l}\sum_{p,q}\epsilon_{ikl}\epsilon_{jpq}R_{klmn}(g).
\end{align}
If now $*\omega=-\omega$ and $*\omega^{\prime}=-\omega^{\prime}$, from \eqref{generalcurvatureoperator} and \eqref{hodgeomega} we obtain
\begin{align}
\langle \RR \omega,\omega^{\prime}\rangle
&=\left(\alpha_{ij}-\left(\beta_{ij}+\beta_{ji}\right)+\gamma_{ij}\right)\xi_{i}\xi^{\prime}_{j}.
\end{align}
This shows that with the isomorphism defined by \eqref{isomorphism}, we can identify the map $\left(\mathcal{W}^{-}+\frac{1}{24}\mathcal{S}\right):\Lambda^{2}_{-}(T^*M)\mapsto\Lambda^{2}_{-}(T^*M)$ with a bilinear form in $S^{2}\left(T^{*}Y\right)$ such that in the local orthonormal basis  $\{e_{1},e_{2},e_{3}\}$ has components
\begin{align}
\alpha_{ij}-\left(\beta_{ij}+\beta_{ji}\right)+\gamma_{ij}.
\end{align}
with $\alpha_{ij},\beta_{ij},\gamma_{ij}$ given by \eqref{curvatureoperatorcomponents}.
Since the scalar curvature operator $\mathcal{S}$ contributes a pure-trace term, 
we are done. 
\end{proof}
We next give a more detailed description of some of the terms appearing in (\ref{eqwey2}) and for that purpose we will make use of the following notation:  
\begin{itemize}
\item All letters $i,j,k,l,\ldots$ will denote non-zero indices. 
\item Given $H\in S^{2}\left(\Lambda^{2}\left(M\right)\right)$, by $c_{Y}H$ we will mean the map defined as
\beq
\left(c_{Y}H\right)_{jk}=g_Y^{il}H_{ijlk},\nonumber
\eeq  
\item By $\tr_Y(c_{Y}H)$ we will mean $g^{ij}_{Y}(c_{Y}H)_{ij}$,
\item We will use $\dot{g}_Y*\dot{g}_Y$ to denote linear combinations of contractions of $\dot{g}_Y\otimes\dot{g}_Y$ using the metric $g_Y$.
\end{itemize}
\begin{proposition}\label{hodge}
We have the identity
\begin{align*}
\frac{1}{4}\sum_{k,l,u,v}\epsilon_{ikl}\epsilon_{juv}H_{kluv}=-\left(c_{Y}H-\frac{1}{2}\tr_Y\left(c_{Y}H\right)g_Y\right)_{ij}.
\end{align*} 
\end{proposition}
\begin{proof} Suppose $i=j$, and let $p,q$ with $p<q$ be indices such that $\{1,2,3\}=\{i,p,q\}$
\begin{align*}
\frac{1}{4}\sum_{k,l}\sum_{u,v}\epsilon_{ikl}\epsilon_{juv}H_{kluv}=\frac{1}{4}\sum_{k,l}\sum_{u,v}\epsilon_{ikl}\epsilon_{iuv}H_{klmn}=H_{pqpq}.
\end{align*} 
Note that the trace of $c_{Y}H$ is given by
\begin{align*}
\frac{1}{2}\tr_Y\left(c_{Y}H\right)=H_{pqpq}+H_{ipip}+H_{iqiq}.
\end{align*}
Therefore
\begin{align*}
H_{pqpq}&=\frac{1}{2}\tr_Y(c_{Y}H)-(H_{ipi}+H_{iqiq})=\frac{1}{2}\tr_Y(c_{Y}H)-(c_{Y}H)_{ii}\\
&=\frac{1}{2}\tr_Y(c_{Y}H)\delta_{ii}-(c_{Y}H)_{ii}=-\left(c_{Y}H-\frac{\tr_Y(c_{Y}H)}{2}g_Y\right)_{ii}.
\end{align*}
If now $i\ne j$ and $p$ is such that $\{1,2,3\}=\{i,j,p\}$, we have
\begin{align*}
&\frac{1}{4}\sum_{k,l}\sum_{u,v}\epsilon_{ikl}\epsilon_{juv}H_{kluv}
=\epsilon_{ijp}\epsilon_{jip}H_{jpjp}\nonumber\\
&= -H_{jpip}=-\left(c_{Y}H\right)_{ij}
=-\left(c_{Y}H-\frac{\tr_Y\left(c_{Y}H\right)}{2}g_Y\right)_{ij},
\end{align*}
and the claim follows.
\end{proof}
We will also need to compute the Christoffel symbols and components of the curvature tensor of $g$ in terms of the metric $g_Y$: 
\begin{proposition}\label{christoffel}
The Christoffel symbols of the metric $g=dt^2+g_Y$ are given by
\begin{align*}
\Gamma_{i0}^{k}(g)&=\frac{1}{2}g_Y^{kl}\left(\dot{g}_Y\right)_{il},~
\Gamma_{ij}^{0}(g)=-\frac{1}{2}\left(\dot{g}_Y\right)_{ij},~\Gamma_{ij}^{k}(g)=\Gamma_{ij}^{k}(g_Y),\\
\Gamma_{00}^{k}(g)&=\Gamma_{0i}^{0}(g)=0.
\end{align*}
For the components of the curvature tensor we have 
\begin{align*}
R_{0ij}^{0}&=-\frac{1}{2}\left(\ddot{g}_{Y}\right)_{ij}+(\dot{g}_Y*\dot{g}_Y)_{ij},\\
R^{k}_{ijl}(g)&=R^{k}_{ijl}(g_Y)+(\dot{g}_Y*\dot{g}_Y)^k_{ijl}.
\end{align*}
In particular, if $g_Y$ is independent of $t$, then 
\begin{align}
\Gamma_{\alpha\beta}^{\gamma}(g)=\left\{
\begin{array}{ll}
0 & \text{if any}~\alpha,\beta,\gamma~\text{equals}~0\\ 
\Gamma_{\alpha\beta}^{\gamma}(g_{Y})& \text{otherwise}
\end{array}
\right.,
\end{align}
and consequently
\begin{align}\label{curvaturecomponentscylindrical}
R_{\alpha\beta\mu}^{\nu}=\left\{
\begin{array}{ll}
0&\text{if any}~\alpha,\beta,\mu,\nu~\text{equals}~0\\
R_{\alpha\beta\mu}^{\nu}(g_{Y})&\text{otherwise}
\end{array}
\right..
\end{align}
\end{proposition}
\begin{proof} The proof follows from a straightforward computation.
\end{proof}
We can now write out a more convenient expression for $\Omega_{ij}$ in (\ref{eqwey2})
\begin{proposition}\label{omega} The term $\Omega_{ij}$ in \eqref{eqwey2} has the form
\begin{align}
\label{omform}
\Omega_{ij}=\left(-E(g_Y)+\dot{g}_Y*\dot{g}_Y\right)_{ij},
\end{align}
where $E(g_Y)$ is the traceless Ricci tensor of $g_Y$. 
\end{proposition}
\begin{proof} Recall that $\Omega_{ij}$ is given by
\begin{align}
\Omega_{ij}=\tf\left(\frac{1}{4}\sum_{k,l,m,n}\epsilon_{ikl}\epsilon_{jmn}R_{klmn}(g)\right),
\end{align}
and from Proposition \ref{hodge} we must have
\begin{align}
\Omega_{ij}=-\tf\left(c_{Y}Rm(g)-\frac{1}{2}\tr_Y(c_{Y}Rm(g))g_Y\right).
\end{align}
With the expressions obtained for the components of $Rm$ in Proposition \ref{christoffel} we have
\begin{align*}
-\left(c_{Y}Rm(g)\right)_{ij}=\left(-Ric(g_Y)+\dot{g}_Y*\dot{g}_Y\right)_{ij},
\end{align*}
and then
\begin{align*}
\tr_Y\left(c_{Y}Rm(g)\right)=R_{g_Y}+\dot{g}_Y*\dot{g}_Y,
\end{align*}
which implies \eqref{omform}.
\end{proof}

%%%%%%%%%%%%%%%%%%%%%%%%%%%%%%%%%%%%%%%%%%%%%%%%%%%%%%%%%%%%%%%%%%%%%%%%%%%%%%%%%%%%%%%%%%%%%%%%%%%%%%

\section{Linearization of $W^{-}$ at a cylindrical metric}\label{linweylcyl}

\noindent
Consider the cylindrical metric 
\begin{align}
g=dt^2+g_{Y}\label{alpha3},
\end{align}
defined on $M = \R \times Y$, where $g_Y$ is a fixed metric of constant 
curvature $\kappa = +1, 0,$ or $-1$. We note that $g$ is locally 
conformally flat, and therefore is self-dual.
We are interested in studying the linearization of $W^{-}$ at $g$. Given 
$\tilde{h}\in S^{2}(M)$, we consider a path of metrics $g(\epsilon)$ 
with $\epsilon\in (-\delta,\delta)$ for some $\delta>0$ satisfying $g(0)=g$ and 
$g'(0) = \tilde{h}$. 
The linearization of $W^{-}$ at $g$ in the direction of $\tilde{h}$ is the map
\begin{align*}
\left(W^{-}\right)^{\prime}_{g}(\tilde{h})=\frac{\partial}{\partial\epsilon}W^{-}(g_{\epsilon})|_{\epsilon=0}.
\end{align*}
We next define a Dirac-type operator: 
\begin{definition}{\em
Let $\{e_{1},e_{2},e_{3}\}$ be a local orthonormal basis of $\Gamma\left(TY\right)$.
Then, for any $h\in S^{2}(T^{*}Y)$ the operator $\slashd h$ is given in these coordinates by
\begin{align}
(\slashd h)_{ij}=\mathrm{Sym}_{ij}\left(\sum_{k,l}\epsilon_{ikl}d^{\nabla}h_{klj}\right),
\end{align}
where $(d^{\nabla}h)_{klj}$ is given by $(d^{\nabla}h)_{klj}=\nabla_{k}h_{lj}-\nabla_{l}h_{kj}$ .
}
\end{definition}
We also recall the conformal Killing operator:
\begin{definition}{\em
For an $n$-dimensional manifold $(M^n, g)$, 
the \emph{conformal Killing operator} with respect to the metric $g$ is the map
$\mathcal{K}_g:\Lambda^{1}(T^*M)\rightarrow S^2_0(T^*M)$,
given by
\begin{align*}
\mathcal{K}_{g}(\tilde{\omega})= \mathcal{L}_g (\tilde{\omega})- \frac{2}{n}(\delta\tilde{\omega})g,
\end{align*}
where $ \mathcal{L}_g$ is the Lie derivative operator.
}
\end{definition}
In cylindrical coordinates, a tensor $\tilde{h}\in S^{2}(M)$ can be decomposed as
\begin{align*}
\tilde{h}=h_{00}dt\otimes dt+\alpha\odot dt+h,
\end{align*} 
where $h_{00}\in\Lambda^{0}(M)$, $\alpha\in\Lambda^{1}(T^{*}Y)$ and $h\in S^{2}(T^{*}Y)$, so we will use the notation $\tilde{h}=\{h_{00},\alpha,h\}$. The main result of this section is the following

\begin{theorem}\label{mainprop}
For the cylindrical metric given by  $g= dt^2 + g_{Y}$,  the linearization $\left(W^{-}\right)^{\prime}_{g}(\tilde{h})$ with $tr_g(\tilde{h}) = 0$, is given by
\begin{align}
\left(W^{-}\right)^{\prime}(h_{00},\alpha,h)=
\frac{1}{2}\mathcal{K}_{g_{Y}}\left(-\frac{1}{2}dh_{00}+\dot{\alpha}-*d\alpha\right)-\frac{1}{2}\tf(\ddot{h})+\frac{1}{2}{\slashd}\dot{h}-E^{\prime}(h)\label{linweylricci},
\end{align}
where $E^{\prime}(h)$ is the linearization of the traceless Ricci tensor at $g_{Y}$. Equivalently, after computing $E^{\prime}(h)$ explicitly, $\left(W^{-}\right)^{\prime}_{g}(\tilde{h})$ is given by 
\begin{align}
\left(W^{-}\right)^{\prime}(h_{00},\alpha,h)=\frac{1}{2}\mathcal{K}_{g_{Y}}\left(-\frac{1}{2}dh_{00}-\delta_{Y} h+\dot{\alpha}-*d\alpha+\frac{1}{2}d\tr_{Y}(h)\right)\nonumber\\
-\frac{1}{2}\tf(\ddot{h})- \kappa \cdot \tf(h)+\frac{1}{2}{\slashd}\dot{h}
+\frac{1}{2}\Delta_{g_{Y}} \tf(h)
.\label{linweyldelta}
\end{align}
where $\Delta_{g_{Y}}$ is the rough laplacian on $S^{2}(T^{*}Y)$, and 
$(\delta_Yh)_j = \nabla_Y^i h_{ij}$ is the divergence.
\end{theorem}

The remainder of the section will be concerned with the 
proof of Theorem \ref{mainprop}.

\subsection{Conformal Killing operator and $\slashd$}\label{ckos} 
The operator $\slashd$ enjoys the following properties:
\begin{proposition}\label{aleph2}
For the operator $\slashd$, we have
\begin{align}
&\slashd : S^{2}(T^{*}Y) \rightarrow S_{0}^{2}(T^{*}Y), \\
&\slashd \left(ug_{Y}\right)=0 ~\mbox{for any}~ u\in C^{2}(Y)\label{slashconformal}, \\
&\slashd :  S^{2}_0(T^{*}Y) \rightarrow S_{0}^{2}(T^{*}Y) \mbox{ is formally self-adjoint}. 
\end{align}
\end{proposition}
\begin{proof} For the first property, in an orthonormal basis
\begin{align*}
\tr_{Y}\left(\slashd h\right)&=\sum_{i,j}\sum_{k,l}\delta_{ij}\mathrm{Sym}_{ij}\left(\epsilon_{ikl}d^{\nabla}h_{klj}\right)
=\sum_{i=1}^{3}\sum_{k,l\ne i}\epsilon_{ikl}\left(\nabla_{k}h_{li}-\nabla_{l}h_{ki}\right)\\
&=\sum_{i=1}^{3}\sum_{k,l\ne i}\epsilon_{ikl}\nabla_{k}h_{li}-\sum_{i=1}^{3}\sum_{k,l\ne i}\epsilon_{ilk}\nabla_{l}h_{ki}=0.
\end{align*}

 For \eqref{slashconformal}, 
in an orthonormal basis and using that $g_{Y}$ is parallel we have
\begin{align*}
\slashd\left(ug_{Y}\right)_{ij}&=\sum_{k,l}\mathrm{Sym}_{ij}\left(\epsilon_{ikl}d^{\nabla}\left(ug_{Y}\right)_{klj}\right)\\
&=\mathrm{Sym}_{ij}\left(\sum_{k,l}\left(\epsilon_{ikl}(\nabla_{k}u)(g_{Y})_{lj}-\epsilon_{ikl}(\nabla_{l}u)(g_{Y})_{kj}\right)\right)\\
&=\mathrm{Sym}_{ij}\left(\sum_{k=1}^{3}\epsilon_{ikj}\nabla_{k}u-\sum_{l=1}^{3}\epsilon_{ijl}\nabla_{l}u\right)
=-2\mathrm{Sym}_{ij}\left(\sum_{k=1}^{3}\epsilon_{ijk}\nabla_{k}u\right),
\end{align*}
and since $\epsilon_{ijk}\nabla_{k}u$ is skew-symmetric in $i,j$, it follows that $\slashd\left(ug_{Y}\right)=0$.  

 Finally, let $h,h^{\prime}$ be elements in $S^{2}(T^{*}Y)$, then in an orthonormal basis we have
\begin{align*}
&\int_{Y}\langle \slashd h, h^{\prime}\rangle dV_{g_{Y}}=\frac{1}{2}\sum_{i,j}\sum_{k,l}\int_{Y}\epsilon_{ikl}\left(\nabla_{k}h_{lj}-\nabla_{l}h_{kj}\right)h^{\prime}_{ij}dV_{g_{Y}}\nonumber\\
&=\sum_{i,j}\sum_{k,l}\int_{Y}\epsilon_{ikl}\nabla_{k}h_{lj}h^{\prime}_{ij}dV_{g}
=-\sum_{i,j}\sum_{k,l}\int_{Y}\epsilon_{ikl}h_{lj}\nabla_{k}h^{\prime}_{ij}dV_{g_{Y}}\\
&=\sum_{i,j}\sum_{k,l}\int_{Y}\epsilon_{lki}\nabla_{k}h^{\prime}_{ij}h_{lj}dV_{g_{Y}}
=\int_{Y}\langle h,\slashd h^{\prime}\rangle dV_{g_{Y}}.
\end{align*}
\end{proof}
For the operators $\slashd,\mathcal{K}_{g}$ and $\mathcal{D}$ we have 
\begin{proposition}\label{slashdK} The operators $\slashd$ and $\mathcal{K}_{g}$ satisfy the following identities
\begin{align}
\slashd \mathcal{L}_{g_Y} (\omega) &=\mathcal{K}_{g_{Y}}(\tilde{*}d\omega)\label{alpha5},\\
\mathcal{D}\left(\mathcal{K}_{g_{Y}}(\tilde{\omega})\right)&=0 ~\text{for any}~\tilde{\omega}\in\Lambda^{1}(T^{*}M).\label{alpha6}
\end{align}
\end{proposition}
\begin{proof} 
Identity (\ref{alpha5}) is a consequence of the following computation: let 
$h=\mathcal{L}_{g_{Y}}(\omega)$ then
\begin{align}
\left(\slashd h\right)_{ij}
&=\mathrm{Sym}_{ij}\left(\sum_{k,l}\epsilon_{ikl}\left(\nabla_{k}h_{lj}-\nabla_{l}h_{kj}\right)\right)\nonumber\\
&=\mathrm{Sym}_{ij}\left(\sum_{k,l}\epsilon_{ikl}\left(\nabla_{k}\nabla_{l}\omega_{j}+\nabla_{k}\nabla_{j}\omega_{l}-\nabla_{l}\nabla_{k}\omega_{j}-\nabla_{l}\nabla_{j}\omega_{k}\right)\right).\label{alpha9}
\end{align} 
Commuting covariant derivatives in \eqref{alpha9} we obtain
\begin{align}\label{alpha9_1}
\left(\slashd h\right)_{ij}=\mathrm{Sym}_{ij}\left(\sum_{k,l}\epsilon_{ikl}\left(\nabla_{j}\nabla_{k}\omega_{l}-\nabla_{j}\nabla_{l}\omega_{k}-R_{klj}^{p}\omega_{p}-R_{kjl}^{p}\omega_{p}+R_{ljk}^{p}\omega_{p}\right)\right).
\end{align}
Note that $-R_{klj}^{p}-R_{kjl}^{p}+R_{ljk}^{p}=-2R^{p}_{klj}$ by the algebraic Bianchi identity, so \eqref{alpha9_1} becomes
\begin{align}
\left(\slashd h\right)_{ij}=\mathrm{Sym}_{ij}\left(\sum_{k,l}\epsilon_{ikl}\left(\nabla_{j}\nabla_{k}\omega_{l}-\nabla_{j}\nabla_{l}\omega_{k}\right)-2\sum_{k,l}\epsilon_{ikl}R^{p}_{klj}\omega_{p}\right).
\end{align}
Since $g_{Y}$ has constant sectional curvature equal to $\kappa$
\begin{align*}
&-2\sum_{k,l}\epsilon_{ikl}R_{klj}^{p} \omega_p=-2 \kappa \sum_{k,l}\epsilon_{ikl}\left(\delta_{k}^{p}\delta_{lj}-\delta_{l}^{p}\delta_{kj}\right)\omega_{p}\\
&=-2 \kappa \left(\sum_{k,l}\epsilon_{ikl}\omega_{k}\delta_{lj}-\sum_{k,l}\epsilon_{ikl}\omega_{l}\delta_{kj}\right)
=-2 \kappa \left(\sum_{k}\left(\epsilon_{ikj}-\epsilon_{ijk}\right)\omega_{k}\right).
\end{align*}
Since $\epsilon_{ikj}-\epsilon_{ijk}$ is skew-symmetric in $i,j$ we obtain
\begin{align*}
\left(\slashd h\right)_{ij}&=\mathrm{Sym}_{ij}\left(\sum_{k,l}\epsilon_{ikl}\left(\nabla_{j}\nabla_{k}\omega_{l}-\nabla_{j}\nabla_{l}\omega_{k}\right)\right)\\
&=\mathrm{Sym}_{ij}\left(\nabla_{j}\left(\sum_{k,l}\epsilon_{ikl}\left(\nabla_{k}\omega_{l}-\nabla_{l}\omega_{k}\right)\right)\right)\\
&=2\mathrm{Sym}_{ij}\left(\nabla_{j}\left(\tilde{*}d\omega\right)_{i}\right)=\nabla_{j}\left(\tilde{*}d\omega\right)_{i}+\nabla_{i}(\tilde{*}d\omega)_{j}\\
&=\left(\mathcal{L}_{g_Y}(\tilde{*}d\omega)\right)_{ij}.
\end{align*}
Since $\slashd h$ is traceless, we actually obtain 
$\slashd \mathcal{L}_{g_Y}(\omega) =\mathcal{K}_{g_Y}(\tilde{*}d\omega)$ as needed. 
For proving (\ref{alpha6}), we note that by diffeomorphism invariance of $W^{-}$ and since $g$ is locally conformally flat we have $\mathcal{D}(\mathcal{L}_g(\tilde{\omega}))=0$ for any 1-form $\tilde{\omega}\in\Lambda^{1}(T^{*}M)$. By the conformal invariance of 
$W^{-}$, we have $\mathcal{D}(fg)=0$ for any $f\in C^{\infty}(M)$, therefore the composition of $\mathcal{D}$ and $\mathcal{K}_{g}$ is zero. 
\end{proof}
%%%%%%%%%%%%%%%%%%%%%%%%%%%%%%%%%%%%%%%%%%%%%%%%%%%%%%%%%%%%%%%%%%%%%%%%%%%%%%%%%%%%%%%%%%%

\subsection{The case of no radial components}
We first compute $\left(W^{-}\right)^{\prime}(\tilde{h})$ assuming that $\tilde{h}$ has no radial components, i.e. $\tilde{h}$ has the form $\tilde{h}=\{0,0,h\}$. 
\begin{proposition}\label{linweyltangential} The linearization of $W^{-}$ at $g=dt^2+g_{Y}$ in the direction $\tilde{h}=\{0,0,h\}$ is 
\begin{align}
\left(W^{-}\right)^{\prime}(\tilde{h})=
-\frac{1}{2}\tf(\ddot{h})+\frac{1}{2}\left(\slashd\dot{h}\right)-E_{g_{Y}}^{\prime}(h).
\end{align}
\end{proposition}

\begin{proof} 
We start by linearizing the component $\Omega_{ij}$ in \eqref{eqwey2}. 
Note that \begin{align}\label{quadraticzero}
\frac{\partial}{\partial\epsilon}\left(\dot{g}_Y(\epsilon)*\dot{g}_Y(\epsilon)\right)|_{\epsilon=0}=0,
\end{align}
for any variation which is purely spherical, that is, a variation which only 
deforms the cross-section metric on $Y$.  From Proposition \ref{omega} and  \eqref{quadraticzero} 
it is clear that for $\tilde{h}=\{0,0,h\}$ we have
\begin{align}\label{linOmega}
\Omega_{ij}^{\prime}(\tilde{h})=-E^{\prime}(h).
\end{align}
For the term $\Phi_{ij}$ in \eqref{eqwey2}, we consider a purely spherical deformation $g_{\epsilon}$ of $g$ in the direction of $h$ so that from \eqref{phi} we have
\begin{align*}
\Phi_{ij}(g_{\epsilon})=\tf_{g_Y(\epsilon)}\left(R_{0i0j}(dt^2 +g_Y(\epsilon))\right),
\end{align*}
and from Proposition \ref{christoffel}
\begin{align}
R_{0i0j}(g_{\epsilon})=\left(-\frac{1}{2}\ddot{g_{Y}}(\epsilon)+\dot{g}_Y(\epsilon)*\dot{g}_Y(\epsilon)\right)_{ij},
\end{align}
then 
\begin{align}\label{linPhi}
\left(\Phi^{\prime}_{g}\right)_{ij}(\tilde{h})=\frac{\partial}{\partial\epsilon}\left(-\tf_{g_Y}(\ddot{g}_{Y}(\epsilon))+\dot{g}_Y(\epsilon)*\dot{g}_Y(\epsilon)\right)|_{\epsilon=0}=-\frac{1}{2}\tf_{g_{Y}}\ddot{h}.
\end{align}
Finally, for the components $\Psi_{ij}$ we recall that we can express 
$\Psi_{ij}(dt^2+g_Y)$ as 
\begin{align*}
\Psi_{ij}=\mathrm{Sym}_{ij}\left(\sum_{k,l}\epsilon_{jkl}(g_Y)R_{0ikl}(dt^2+g_Y)\right).
\end{align*}
Note that taking the tracefree part is not necessary, see Proposition \ref{aleph2}. 
Before linearizing $\epsilon_{jkl}(g_Y)R_{0ikl}(dt^2+g_Y)$, we note that if we evaluate $\Psi_{ij}$ along a purely spherical deformation $g_{\epsilon}$ of $g$ in the direction of $h$, the symbol $\epsilon_{jkl}$ may depend on $g_Y(\epsilon)$ and so we must write
\begin{align*}
\Psi_{ij}(g_{\epsilon})=\mathrm{Sym}_{ij}\left(\sum_{k,l}\epsilon_{jkl}(g_Y(\epsilon))R_{0ikl}(g_{\epsilon})\right),
\end{align*}
however, since $R_{0ijk}(dt^2+g_{Y})=0$ for all choices of $i,j,k$ as seen in \eqref{curvaturecomponentscylindrical}, we conclude that the linearization of $\Psi_{ij}$ in the direction $\tilde{h}=\{0,0,h\}$ is
\begin{align*}
\mathrm{Sym}_{ij}\left(\sum_{k,l}\epsilon_{jkl}\left(R^{\prime}_{g}\right)_{0ikl}(\tilde{h})\right).
\end{align*}     
Linearizing $Rm$ at $g$ in the direction of $\tilde{h}$, and using 
Proposition \ref{christoffel}, we obtain 
\begin{align}\label{linriem}
\left(R^{\prime}_{g}\right)_{0ikl}(\tilde{h})=\frac{1}{2}\left(\nabla_{0}\nabla_{l}\tilde{h}_{ik}-\nabla_{0}\nabla_{k}\tilde{h}_{il}-\nabla_{i}\nabla_{l}\tilde{h}_{0k}+\nabla_{i}\nabla_{k}\tilde{h}_{0l}\right).
\end{align}
It is easy to see that
\begin{align*}
\nabla_{0}\nabla_{k}\tilde{h}_{il}=\nabla_{k}\dot{h}_{il},  \mbox{ and }
\nabla_{i}\nabla_{k}\tilde{h}_{0l} =\nabla_{i}\nabla_{l}h_{0j}=0,
\end{align*}
so we have proved
\begin{align}
\begin{split}
\label{linPsi}
\Psi_{ij}^{\prime}(\tilde{h})&=\mathrm{Sym}_{ij}\left(\sum_{k,l}\epsilon_{jkl}\left(R^{\prime}_{g}\right)_{0ikl}(\tilde{h})\right)\\
&=-\frac{1}{2}\mathrm{Sym}_{ij}\left(\sum_{k,l}\epsilon_{jkl}\left(\nabla_{k}\dot{h}_{il}-\nabla_{l}\dot{h}_{ik}\right)\right) =-\frac{1}{2}(\slashd\dot{h})_{ij}.
\end{split}
\end{align}
The proposition follows from combining  \eqref{linOmega}, \eqref{linPhi} and \eqref{linPsi}.
\end{proof}

%%%%%%%%%%%%%%%%%%%%%%%%%%%%%%%%%%%%%%%%%%%%%%%%%%%%%%%%%%%%%%%%%%%%%%%%%%%%%%%%%%%%%%%%%%%
\subsection{The case of conformal variations}\label{confvariations}
Using conformal invariance, we next extend the formula in Proposition \ref{linweyltangential}
to tensors of the form $\{h_{00}, 0, h\}$. 
\begin{proposition}\label{conformal}
The linearization of $W^{-}$ at $g$ in the direction $\tilde{h} = \{h_{00},0,h\}$ is 
\begin{align*}
\mathcal{D}(h_{00}dt\otimes dt+h)=\mathcal{D}(h)-\frac{1}{2}\left(\nabla^{2}_{Y}h_{00}-\frac{1}{3}(\Delta_{g_{Y}}h_{00}) g_Y\right).
\end{align*}
\end{proposition}
\begin{proof}Since the cylinder is locally conformally flat, 
for any $C^{2}$ function $v$ we have
\begin{align*}
\mathcal{D}\left(v\left(dt^2+g_{Y}\right)\right)=0,
\end{align*}
therefore
\begin{align*}
&\mathcal{D}\left(h_{00}dt\otimes dt+h\right)
=\mathcal{D}\left(h_{00}\left(dt^2+g_{Y}\right)-h_{00}g_{Y}+h\right)\\
&=\mathcal{D}(h-h_{00}g_{Y})=\mathcal{D}(h)-\mathcal{D}(h_{00}g_{Y}).
\end{align*}  
Since $h_{00}g_{Y}$ is a scalar tensor we have by Corollary \ref{linweyltangential} and \eqref{slashconformal}
\begin{align} \label{linricciconformal}
\mathcal{D}(h_{00}g_{Y})=-E_{g_{Y}}^{\prime}\left(h_{00} g_{Y}\right)=-E^{\prime}_{g_{Y}}(h_{00}g_{Y}).
\end{align}
Next, consider a  path $\{g_{s}\}$ of metric on $Y$ given by 
$g_{s}=e^{su}g_{Y}$, then $g_{0}=g_{Y}$ and $\partial_{s}g_{s}|_{s=0}=ug_{Y}$.
Since $g_Y$ is Einstein, a standard formula for conformal changes gives
\begin{align*}
E(g_{s})&=-\frac{s}{2}\left(\nabla^{2}_{g_{Y}}u-\frac{1}{3}(\Delta_{g_{Y}}u)g_Y\right)
+\frac{s^{2}}{4}\left(du\otimes du-\frac{1}{3}|\nabla_{g_{Y}} u|^{2}g_{Y}\right).
\end{align*}
Differentiating at $s=0$, we obtain
\begin{align*}
E^{\prime}_{g_{Y}}(ug_{Y})=-\frac{1}{2}\left(\nabla^{2}_{g_{Y}}u-\frac{1}{3}(\Delta_{g_{Y}}u)g_Y\right),
\end{align*}
and the proposition follows. 
\end{proof}

%%%%%%%%%%%%%%%%%%%%%%%%%%%%%%%%%%%%%%%%%%%%%%%%%%%%%%%%%%%%%%%%%%%%%%%%%%%%%%
\subsection{Completion of proof of Theorem \ref{mainprop}}\label{mixedcomponents}

Consider now a variation $\tilde{h}$ of the form $\tilde{h}=\{0,\alpha,0\}$. 
\begin{proposition}\label{0a0}
The linearization of $W^{-}$ at $g$ in the direction $ \{0,\alpha,0\}$ is 
given by
\begin{align}\label{linweylmixed}
\mathcal{D}(\{0,\alpha,0\})
&=\frac{1}{2}\mathcal{K}_{g_{Y}}\left(\dot{\alpha}-\tilde{*}d\alpha\right).
\end{align}
\end{proposition}
\begin{proof}
Choose $\omega$ so that $\dot{\omega}=\alpha$. In this case the conformal Killing operator equals
\begin{align*}
\mathcal{K}_{g}(\omega)=\left\{-\frac{1}{2}\delta_{Y}\omega,\alpha,
\mathcal{K}_{g_{Y}}\left(\omega\right)+\left(\frac{1}{6}\delta_{Y}\omega\right)g_{Y}\right\}.
\end{align*}
We write
\begin{align*}
\mathcal{D}(\{0,\alpha,0\})&=\mathcal{D}(\{0,\alpha,0\}-\mathcal{K}_{g}(\omega))\\
&=\mathcal{D}\left(\{0,\alpha,0\}-\left\{-\frac{1}{2}\delta_{Y}\omega,\alpha, \mathcal{K}_{g_{Y}}\left(\omega\right)+\left(\frac{1}{6}\delta_{Y}\omega\right)g_{Y}\right\}\right)\\
&=\mathcal{D}\left(\frac{1}{2}\delta_{Y}\omega,0,-\mathcal{K}_{g_{Y}}\left(\omega\right)-\left(\frac{1}{6}\delta_{Y}\omega\right)g_{Y}\right).
\end{align*} 
Recall that for any $C^{2}$ function $u$ we have 
$\mathcal{D}\left(u dt^2\right)=\mathcal{D}\left(-u g_{Y}\right)$, using 
\eqref{linricciconformal} we obtain
\begin{align}
\mathcal{D}(\{0,\alpha,0\})=
\mathcal{D}\left(-\mathcal{K}_{g_{Y}}\left(\omega\right)-\left(\frac{1}{6}\delta_{Y}\omega\right)g_{Y}-\left(\frac{1}{2}\delta_{Y}\omega\right)g_{Y}\right)
=\mathcal{D}\left(-\mathcal{L}_{g_Y}(\omega) \right) 
\label{Dmixedcomponents}
\end{align}
From Corollary \ref{linweyltangential} and from \eqref{alpha5} and \eqref{Dmixedcomponents} we obtain
\begin{align}\label{floertrick}
\mathcal{D}(\{0,\alpha,0\})=\mathcal{K}_{g_{Y}}\left(\frac{1}{2}\ddot{\omega}\right)-\frac{1}{2}\slashd \mathcal{L}_{g_Y}(\dot{\omega}),
\end{align}
and since $\dot{\omega}=\alpha$, we have
\begin{align}
\ddot{\omega}&=\dot{\alpha},\label{omegadoubledot}\\
\slashd \mathcal{L}_{g_Y}(\dot{\omega})
&=\mathcal{K}_{g_{Y}}(\tilde{*}d\dot{\omega})=\mathcal{K}_{g_{Y}}(\tilde{*}d\alpha),\label{confkillingdotomega}
\end{align}
so from \eqref{floertrick}, \eqref{omegadoubledot} and \eqref{confkillingdotomega} we obtain
\eqref{linweylmixed}.
\end{proof}
With \eqref{linweylmixed} we are ready to prove Theorem \ref{mainprop}.
\begin{proof}[Proof of Theorem \ref{mainprop}]
Combining Corollary \ref{linweyltangential}, Proposition \ref{conformal} and \eqref{linweylmixed} we obtain \eqref{linweylricci}. 
In order to prove \eqref{linweyldelta}  we linearize $E$ at $g_{Y}$ in the direction of $h$  
\begin{align*}
E^{\prime}(h)&=\left(Ric(g)-\frac{1}{3}R_{g}g\right)_{g_{Y}}^{\prime}(h)
=Ric^{\prime}_{g_{Y}}(h)-\frac{1}{3}R_{g_{Y}}^{\prime}(h)g_{Y}-\frac{1}{3}R_{g_{Y}}h.
\end{align*}
The linearization of $Ric$ is
\begin{align}\label{linricci}
\left(Ric^{\prime}_{g_{Y}}(h)\right)_{ij}=-\frac{1}{2}\Delta_{L} h_{ij}-\frac{1}{2}\nabla^{2}_{ij}\tr(h)+\frac{1}{2}\left(\nabla_{i}\delta_{j}h+\nabla_{j}\delta_{i}h\right),
\end{align}
where $\Delta_{L}h$ is the \emph{Lichnerowicz} Laplacian given by
\begin{align}\label{lichnerowicz}
\Delta_{L}h_{ij}=\Delta_{g_{Y}} h_{ij}+2R_{iljp}h^{lp}-R_{i}^{p}h_{jp}-R_{j}^{p}h_{ip}.
\end{align}
Since $g_{Y}$ has constant sectional curvature $\kappa$, $\Delta_{L}$ can be computed as
\begin{align}
\Delta_{L}h_{ij}&=\Delta_{g_{Y}}h_{ij}+2 \kappa \left(\left(g_{Y}\right)_{ij}\left(g_{Y}\right)_{lp}-\left(g_{Y}\right)_{ip}\left(g_{Y}\right)_{lj}\right)h^{lp}-2 \kappa \delta_{i}^{p}h_{jp}
-2\kappa\delta_{j}^{p}h_{ip}\nonumber\\
&=\Delta_{g_{Y}} h_{ij}+2\kappa\tr_{Y}(h)\left(g_{Y}\right)_{ij}-2 \kappa h_{ij}
-4 \kappa h_{ij}\nonumber\\
&=\Delta_{g_{Y}} h_{ij}-6 \kappa\tf(h)_{ij}.\label{lichnerowiczsect1}
\end{align}
On the other hand, the linearization of $R_{g}$ is
\begin{align}\label{linscalar}
R^{\prime}(h)=-\Delta_{g_{Y}}\tr(h)+\delta_{Y}\delta_{Y} h-\langle Ric(g_{Y}),h\rangle_{g_Y}.
\end{align}
and 
\begin{align}\label{scalarsect1}
\frac{1}{3}(R_{g_{Y}}h-\langle Ric(g_{Y}),h\rangle_{g_Y} g_{Y})=2 \kappa\tf(h).
\end{align}
Combining \eqref{linricci} and \eqref{linscalar}, we conclude that $E^{\prime}_{g_{Y}}(h)$ is given by
\begin{align}
E^{\prime}_{g_{Y}}(h)&=-\frac{1}{2}\left(\Delta_{L}h+\nabla^{2}\tr(h)-\frac{2}{3}\Delta_{g_{Y}}\tr_{Y}(h)g_{Y}\right)\nonumber\\
&+\frac{1}{2}\left(\mathcal{L}_{g_{Y}}(\delta_Y h)-\frac{1}{3}(\delta_{Y}\delta_{Y} h)g_{Y}\right)-\frac{1}{3}(R_{g_{Y}}h-\langle Ric(g_{Y}),h\rangle_{g_Y}),\label{lintracelessricci}
\end{align}
and using \eqref{lichnerowiczsect1} and \eqref{scalarsect1}, we finally obtain
\begin{align}\label{lintracelessriccisect1}
E^{\prime}_{g_{Y}}(h)=-\frac{1}{2}\left(\Delta_{g_{Y}} \tf(h)+\stl{\circ}{\nabla^{2}}\tr_Y(h)\right)+\frac{1}{2}\mathcal{K}_{g_{Y}}\left(\delta_{Y} h\right) +\kappa \cdot \tf(h),
\end{align}
where $\stl{\circ}{\nabla^{2}}$ denotes the traceless Hessian operator.
From \eqref{linweylricci} and \eqref{lintracelessriccisect1}, \eqref{linweyldelta} follows easily.
\end{proof}

%%%%%%%%%%%%%%%%%%%%%%%%%%%%%%%%%%%%%%%%%%%%%%%%%%%%%%%%%%%%%%%%%%%%%%%%%%%%%%%%%%%%
\section{Some properties of $\slashd$} 
\label{slashdsec}

In this section we derive several useful identities for the operator $\slashd$ introduced in Section \ref{linweylcyl} apart from those proved in Subsection \ref{ckos}.
First, we have a crucial formula for the square of $\slashd$:
\begin{proposition}
\label{squaredprop}
The operator  $\slashd^{2}:S^{2}(T^{*}Y)\mapsto S^{2}_{0}(T^{*}Y)$ is given by 
\begin{align}
\label{squared}
\slashd^{2}h=-4\Delta_{g_{Y}}\tf(h)-2\stl{\circ}{\nabla^2} \tr_{Y}(h)+3\mathcal{K}_{g_{Y}}(\delta_{Y} h)+12 \kappa \cdot \tf(h).
\end{align}
\end{proposition}
\begin{proof} The proof is moved to Appendix \ref{appendix}.
\end{proof}
Next, we have
\begin{proposition}\label{deltaslashd}
For any $h\in S^{2}(T^{*}Y)$ we have $\delta_{Y}\left(\slashd h\right)=\tilde{*}d\delta_{Y}h.$
\end{proposition}
\begin{proof}
In a local orthonormal basis we have
\begin{align}\label{divslashd}
\left(\delta_{Y}\left(\slashd h\right)\right)_{i}&=\sum_{j=1}^{3}\nabla_{j}(\slashd h)_{ij}\nonumber\\
&=\sum_{j=1}^{3}\sum_{k,l}\epsilon_{ikl}\nabla_{j}\nabla_{k}h_{lj}+\sum_{j=1}^{3}\sum_{p,q}\epsilon_{jpq}\nabla_{j}\nabla_{p}h_{qi}.
\end{align}
Commuting covariant derivatives we have
\begin{align*}
&\sum_{j=1}^{3}\sum_{k,l}\epsilon_{ikl}\nabla_{j}\nabla_{k}h_{lj}=\sum_{j=1}^{3}\sum_{k,l}\epsilon_{ikl}\left(\nabla_{k}\nabla_{j}h_{lj}-R_{jkl}^{s}h_{sj}-R^{s}_{jkj}h_{ls}\right)\\
&=\sum_{j=1}^{3}\sum_{k,l}\epsilon_{ikl}\left(\nabla_{k}\nabla_{j}h_{lj} +
\kappa \big( - h_{jj}g_{kl}+h_{kj}g_{jl}-h_{lj}g_{kj}+h_{lk}g_{jj} \big)\right)\\
&=\left(\tilde{*}d\delta_{Y}h\right)_{i} + \kappa \sum_{k,l}\left(\sum_{j=1}^{3} 
\left\{-h_{jj}\epsilon_{ikl}g_{kl}+\epsilon_{ikl}h_{kl}-\epsilon_{ikl}h_{lk}+3\epsilon_{ikl}h_{lk}
\right\} \right).
\end{align*}
Since all terms in the sum consist of a term skew-symmetric in $k$ and $l$ 
times a term symmetric in $k$ and $l$, the sum is zero,
so we obtain
\begin{align}\label{sumikl}
\sum_{j=1}^{3}\sum_{k,l}\epsilon_{ikl}\nabla_{j}\nabla_{k}h_{lj}=\left(\tilde{*}d\delta_{Y}h\right)_{i}.
\end{align}
We also have
\begin{align*}
&\sum_{j=1}^{3}\sum_{p,q}\epsilon_{jpq}\nabla_{j}\nabla_{p}h_{qi}=\sum_{j=1}^{3}\sum_{p,q}\epsilon_{jpq}\left(\nabla_{p}\nabla_{j}h_{qi}-R^{s}_{jpq}h_{si}-R^{s}_{jpi}h_{qs}\right)\\
&=\sum_{j=1}^{3}\sum_{p,q}\epsilon_{jpq}\left(\nabla_{p}\nabla_{j}h_{qi}
+ \kappa \big( -h_{ij}g_{pq}+h_{pi}g_{qj}+h_{qj}g_{pi}-h_{pq}g_{ji} \big)\right),
\end{align*}
and clearly the last 4 terms sum to zero. So we have
\begin{align*}
\sum_{j=1}^{3}\sum_{p,q}\epsilon_{jpq}\nabla_{j}\nabla_{p}h_{qi} 
= \sum_{j=1}^{3}\sum_{p,q}\epsilon_{jpq}\nabla_{p}\nabla_{j}h_{qi}
\end{align*}
By reindexing $j$ and $p$ on the right hand side, we obtain 
\begin{align*}
\sum_{j=1}^{3}\sum_{p,q}\epsilon_{jpq}\nabla_{j}\nabla_{p}h_{qi} 
= \sum_{j=1}^{3}\sum_{p,q}\epsilon_{pjq}\nabla_{j}\nabla_{p}h_{qi}
= - \sum_{j=1}^{3}\sum_{p,q}\epsilon_{jpq}\nabla_{j}\nabla_{p}h_{qi},
\end{align*}
so this sum vanishes. 
Combining this with \eqref{divslashd} and \eqref{sumikl}, the proposition then follows.
\end{proof}

\begin{corollary} For any $h\in S^{2}(T^{*}Y)$ we have
\begin{align}
\slashd\Delta_{g_{Y}}h=\Delta_{g_{Y}}\slashd h.
\end{align}
\end{corollary}

\begin{proof} From \eqref{squared} we have
\begin{align*}
\slashd^{3}h=-4\slashd\Delta_{g_{Y}}\tf{h}-2\slashd\stackrel{\circ}{\nabla^2}\tr_{Y}(h)+3\slashd\mathcal{K}_{g_{Y}}(\delta_{Y}h)+12 \kappa \cdot \slashd(\tf(h)),
\end{align*} 
and clearly
\begin{align}
-4\slashd\Delta_{g_{Y}}\tf{h}&=-4\slashd\Delta_{g_{Y}}h,\\
\slashd\stackrel{\circ}{\nabla^2}\tr_{Y}(h)&=\frac{1}{2}\slashd\mathcal{K}_{g_{Y}}(d\tr_{g_{Y}}h)=0,\\
3\slashd\mathcal{K}_{g_{Y}}(\delta_{Y}h)&=3\mathcal{K}_{g_{Y}}(\tilde{*}d\delta_{Y}h)=3\mathcal{K}_{g_{Y}}(\delta_{Y}\slashd h),\\
 \slashd(\tf (h))&=\slashd h.
\end{align}
On the other hand we have
\begin{align*}
\slashd^{3}h&=\slashd^{2}\slashd h=-4\Delta_{g_{Y}}\slashd h-2\stl{\circ}{\nabla^{2}}\tr_{Y}(\slashd h)+3\mathcal{K}_{g_{Y}}(\delta_{Y} \slashd h)+12 \kappa \cdot \slashd h\\
 &=-4\Delta_{g_{Y}}\slashd h+3\mathcal{K}_{g_{Y}}(\delta_{Y} \slashd h)+12 \kappa \cdot \slashd h,
 \end{align*}
and this proves the claim.
\end{proof}
For the next lemma we will use $\Delta_{H}$ to denote the Hodge-Laplacian 
on $\Lambda^{1}(T^{*}Y)$  which is given by 
\begin{align*}
\Delta_{H}\omega&=-d\delta_{Y}\omega-\delta_{Y}d\omega\\
&=d\tilde{*}d\tilde{*}\omega-\tilde{*}d\tilde{*}d\omega,
\end{align*} 
which is related to the rough Laplacian on $1$-forms by the Weitzenb\"ock 
formula
\begin{align}
\Delta_{g_Y} = - \Delta_H + 2 \kappa.
\end{align}
\begin{lemma}\label{commutator} The operator $\mathcal{K}_{g_{Y}}(\omega)$ satisfies 
\begin{align*}
\delta_{Y}\mathcal{K}_{g_{Y}}(\omega)&=\Delta_{g_{Y}}\omega+\frac{1}{3}\cdot d\delta_{Y}\omega +2 \kappa \omega\\
&=-\Delta_{H}\omega+\frac{1}{3}\cdot d\delta_{Y} \omega +4 \kappa \omega,
\end{align*}
and also 
\begin{align*}
\Delta_{g_{Y}}\mathcal{K}_{g_{Y}}(\omega)&=\mathcal{K}_{g_{Y}}\left(\left(\Delta_{g_{Y}}
+4 \kappa \right)\omega\right)\\
&=\mathcal{K}_{g_{Y}}\left(\left(-\Delta_{H}+6 \kappa \right)\omega\right).
\end{align*}
\end{lemma}
\begin{proof} Both identities follow from straightforward computations, see for example  \cite[Appendix]{Streets}.
\end{proof}
\begin{corollary}\label{slashdconstant} If $h$ is a divergence-free 
eigentensor of $\Delta_{g_{Y}}$ with eigenvalue $-\lambda$ in $S_{0}^{2}(T^{*}Y)$ 
which  satisfies  $\Delta_{g_Y}(h)=-\lambda\cdot h$ then  
$\slashd h=\pm 2\sqrt{\lambda+3 \kappa}\cdot h$  and if $h$ has the form $h=\mathcal{K}_{g_{Y}}(\omega)$ with $\Delta_{H}\omega=\nu\omega$ and $\delta_{Y}\omega=0$  then 
$\slashd h=\pm \sqrt{\nu}\cdot h$. 
Both signs occur if $(Y^3,g_Y)$ admits an orientation-reversing isometry
(which is always true for $S^3$ or $\kappa = 0$). 
\end{corollary}
\begin{proof}
Note that in either of the above cases  
we must have $\slashd^{2} h=c^{2}\cdot h$  for some constant $c$. To see this, if $\Delta_{g_{Y}}h=-\lambda h$ and $\delta_{Y}h=0$ 
with $\lambda>0$ then by Proposition \ref{squaredprop} we have
\begin{align}
\slashd^{2}h=-4\cdot \Delta_{g_{Y}}h+12 \kappa \cdot h=(4\lambda+12 \kappa)\cdot h.
\end{align}
In this case $c^{2}=(4\lambda+12\kappa)$. For the second case, 
from Lemma \ref{commutator} we have the identities 
\begin{align}
\Delta_{g_{Y}}h=(6\kappa -\nu)\cdot\mathcal{K}_{g_{Y}}(\omega)
=(6 \kappa-\nu)\cdot h,
\end{align}
and 
\begin{align}
\delta_{g_{Y}}\mathcal{K}_{g_{Y}}(\omega)&=(4 \kappa-\nu)\cdot\mathcal{K}_{g_{Y}}(\omega)\\
&=(4\kappa-\nu)\cdot h.
\end{align}
Since we also have 
\begin{align}
\tr_{Y}(h)=0,
\end{align}
we easily obtain from Proposition \ref{squaredprop} 
\begin{align}
\slashd^{2}h=\nu \cdot h,
\end{align}
so in this case $c^{2}=\nu$. 

In both cases, observe that if we fix an eigenvalue $\upsilon$ of $\Delta_{g_{Y}}$ on 
$S^{2}_{0}(T^{*}Y)$ then the eigenspace 
$E_{\upsilon}=\{h\in S^{2}_{0}(T^{*}Y):\Delta_{g_{Y}}h=-\upsilon\cdot h\}$ 
is not necessarily $SO(3)$-irreducible, and decomposes 
into $E_{\upsilon} = A^{+}_{c} \oplus A^{-}_{c}$, where
$A^{\pm}_{c}=\{h\in S^{2}_{0}(T^{*}Y):\slashd h=\pm c\cdot h\}$ and 
are $SO(3)$-invariant.
To see this, given any eigentensor, writing the equation 
$\slashd^{2}h=c^{2}\cdot h$  as
\begin{align}
\left(\slashd+c\cdot I\right)\left(\slashd-c\cdot I\right)h=0,
\end{align}
we conclude that either $+c$ or $-c$ occurs as an eigenvalue of $\slashd$. 
If $Y^3$ admits an orientation-reversing isometry,
then since the operator $\slashd$ changes sign
under reversal of orientation, pulling an eigentensor back along an 
orientation-reversing isometry shows that both $A^{+}_{c}$ and $A^{-}_{c}$ 
are nontrivial and of the same dimension. 
\end{proof}

\begin{corollary}
\label{stard}
If $\omega$ is an eigenform of the Hodge Laplacian on $1$-forms 
with eigenvalue $\nu$ satisfying $\delta_{Y}\omega=0$ then 
$\tilde{*}d\omega=\pm \sqrt{\nu}\cdot\omega$. 
Both signs occur if $(Y^3,g_Y)$ admits an orientation-reversing isometry
(which is always true for $S^3$ or $\kappa = 0$). 
\end{corollary}
\begin{proof}
Obviously, since $\delta_{Y}\omega=0$, then 
\begin{align}
(\tilde{*}d)^2 \omega = \tilde{*}d\tilde{*}d = -\delta_{Y} d \omega = \Delta_H \omega = 
\nu \cdot \omega.
\end{align}
Using a similar argument as in Corollary \ref{slashdconstant},
we conclude that $\tilde{*}d\omega=\pm \sqrt{\nu}\cdot\omega$, 
with both signs occurring 
on $S^3$.  
\end{proof}

%%%%%%%%%%%%%%%%%%%%%%%%%%%%%
\section{The adjoint of $\mathcal{D}$}\label{adjointeq}
%%%%%%%%%%%%%%%%%%%%%%%%%%%%%
The adjoint operator will map from 
\begin{align}
\mathcal{D}^* : S^2_0 ( \Lambda^2_-) \rightarrow S^2_0 (T^*M),
\end{align}
and using the decompositions in Subsection \ref{wey} we will think of this as 
\begin{align}\label{adjtdecomp}
\mathcal{D}^* :S^2_0 ( T^* Y) \rightarrow 
 S^2 ( \nu) \oplus ( \nu \odot T^*(Y)) \oplus S^2 (T^* Y).
\end{align}

\begin{proposition} The adjoint operator is given by 
\begin{align}
\mathcal{D}^{*}Z&= \bigg\{   -\frac{1}{2} \delta_{Y}^2 Z, \frac{1}{2} \delta_{Y} \dot{Z} 
+ \frac{1}{2}\delta_{H} \tilde{*} \delta_{Y} Z , - \frac{1}{2} \ddot{Z} - \kappa Z 
- \frac{1}{2} \slashd \dot{Z}
+ \frac{1}{2} \Delta_{g_{Y}} Z\nonumber\\
& - \frac{1}{2} \mathcal{L}_{g_Y}(\delta_{Y} Z) +\frac{1}{2}(\delta_{Y}^{2}Z)g_{Y} \bigg\}.\label{adjoint}
\end{align}
Where $\delta_{H}$ is the Hodge divergence on forms 
given by $\delta_{H}=d^{*}$.
\end{proposition}
\begin{proof} 
Let $Z\in S^{2}_{0}(T^{*}Y)$, from the decomposition \eqref{s2r4} we can see $ S^{2}_{0}(T^{*}Y)$ as embedded in $S^{2}(T^{*}M)$, so taking $\langle,\rangle$ to be the inner product induced by the cylindrical metric $g$ on $S^{2}(T^{*}M)$  we observe that  since $Z$ is traceless with respect to $g_{Y}$ we have for any $\tilde{h}=h_{00}dt\otimes dt+ dt \odot \alpha +h$,
\begin{align}
\langle \mathcal{D}\tilde{h},Z\rangle=\left\langle\frac{1}{2}
\mathcal{L}_{g_Y}\left(-\frac{1}{2}dh_{00}-\delta_{Y} h+\dot{\alpha}-\tilde{*}d\alpha+\frac{1}{2}d\tr_{Y}(h)\right), Z\right\rangle\nonumber\\
+\left\langle -\frac{1}{2}\ddot{h}-\kappa h+\frac{1}{2}\slashd\dot{h}, Z\right\rangle
+\left\langle\frac{1}{2}\Delta_{g_{Y}} h ,Z\right\rangle.
\end{align}
Formal integration by parts then yields 
\begin{align*}
\int_{0}^{\infty}\int_{Y}\langle \mathcal{D}\tilde{h},Z\rangle dtdV_{g_{Y}}
&=-\frac{1}{2}\int_{0}^{\infty}\left( \left(h_{00} \cdot \delta^{2}_{Y}Z\right)+\langle\alpha,\delta_{Y}\dot{Z}+\delta_{H}\tilde{*}\delta_{Y}Z\rangle \right)dtdV_{g_{Y}}\\
&+\int_{0}^{\infty}\int_{Y} \Big\langle h,\frac{1}{2}\left(\delta^{2}_{Y}Z\right)g_{Y} \Big\rangle dtdV_{g_{Y}}\\
&-\int_{0}^{\infty}\int_{Y} \left( \Big\langle h,\frac{1}{2}\mathcal{L}_{g_Y}(\delta_{Y}Z)-\frac{1}{2}\ddot{Z}-\kappa Z \Big\rangle+\frac{1}{2}\langle \slashd\dot{h},Z\rangle \right)dtdV_{g_{Y}}.
\end{align*}
Note that by the inner product
\begin{align*}
\langle\alpha,\delta_{Y}\dot{Z}+\delta_{H}\tilde{*}\delta_{Y}Z\rangle,
\end{align*}
we mean the usual inner product on 1-forms, however, using the decomposition in~\eqref{adjtdecomp}, we identify a 1-form $\xi$ with the tensor
\begin{align*}
\{0,\xi,0\}=\xi\otimes dt+dt\otimes \xi,
\end{align*}
so we obtain
\begin{align*}
\langle\alpha,\delta_{Y}\dot{Z}+\delta_{H}\tilde{*}\delta_{Y}Z\rangle=\frac{1}{2}\langle\{0,\alpha,0\},\{0,\delta_{Y}\dot{Z}+\delta_{H}\tilde{*}\delta_{Y}Z,0\}\rangle.
\end{align*}
Finally, the proposition follows using that $\slashd $ is formally self-adjoint. 
\end{proof}

\begin{proposition}
\label{kiprop}
We have the decompositions
\begin{align}\label{s2ts3}
S^2_0( T^* Y) =  Ker( \delta_{Y}) \oplus Im ( \mathcal{K}_{g_{Y}}), 
\end{align}
and
\begin{align}
\label{hodge2}
\Lambda^{1}(T^{*}Y) = Im(d) \oplus Ker(d^*).
\end{align}
\end{proposition}
\begin{proof}
Since $\delta_{Y}$ is the formal adjoint of $-\frac{1}{2}\mathcal{K}_{g_{Y}}$, 
\eqref{s2ts3} follows from standard Fredholm theory. 
The Hodge decomposition theorem says that 
\begin{align}
\label{hodge1}
\Lambda^{1}(T^{*}Y)=\mathcal{H}^{1}(T^{*}Y)\oplus \left(d\Lambda^{0}(T^{*}Y)\right)\oplus\left(d^{*}\Lambda^{2}(T^{*}Y)\right),
\end{align}
where $\mathcal{H}^{1}(T^{*}Y)$ is the space of harmonic 1-forms in $\Lambda^{1}(T^{*}Y)$,
and \eqref{hodge2} follows easily from this
since $\mathcal{H}^{1}(T^{*}Y)$ and $d^{*}\Lambda^{2}(T^{*}Y)$ are both 
contained in $Ker(d^*)$.
\end{proof}
Using this decomposition we obtain: 
\begin{corollary}\label{3types} Any time dependent  $Z\in S^{2}_{0}(T^{*}Y)$ can be written uniquely as an infinite linear combination of elements of three types, namely  
\begin{enumerate}
\item Elements of type {\em{I}}:
\begin{align}\label{typeI}
f(t)\cdot \mathcal{K}_{g_{Y}}(d\phi),
\end{align}
where $\phi$ is an eigenfunction of $\Delta_{H}$ on $\Lambda^{0}\left(T^{*}Y\right)$,
\item Elements of type {\emph{II}}:
\begin{align}\label{typeII}
f(t)\cdot \mathcal{K}_{g_{Y}}(\omega),
\end{align}
where $\omega$ is an eigenform of $\Delta_{H}$ on $\Lambda^{1}\left(T^{*}Y\right)$ satisfying  $\delta_{Y}\omega=0$,
\item Elements of type {\em{III}}:
\begin{align}\label{typeIII}
f(t)\cdot B,
\end{align}
where $B$ is an eigentensor of $\Delta_{g_{Y}}$ on $S^{2}_{0}(T^{*}Y)$ satisfying $\delta_{Y}B=0$. 
\end{enumerate}
In all of the three above cases $f(t)$ denotes a real-valued function.
\end{corollary}
From Propositions \ref{divslashd}, \ref{commutator}, and \ref{slashdconstant}, we observe that the image of $\mathcal{D}^{*}$ on an element of type I  has the form
\begin{align}
\mathcal{D}^{*}\left(f(t)\cdot\mathcal{K}_{g_{Y}}(d\phi)\right)=
a_{1}(t)\cdot\phi dt\otimes dt +a_{2}(t)\cdot d\phi\odot dt\nonumber\\
+a_{3}(t)\cdot\mathcal{K}_{g_{Y}}(d\phi)+a_{4}(t) \cdot \phi g_{Y},\label{imtypeI}
\end{align}
where each coefficient $a_{i}$ depends on $f$ and the eigenvalue of $\Delta_{H}$ corresponding to $\phi$. 
On elements of type II  the image of $\mathcal{D}^{*}$ is 
\begin{align}\label{imtypeII}
\mathcal{D}^{*}\left(f(t)\cdot\mathcal{K}_{g_{Y}}(\omega)\right)=b_{1}(t)\cdot\omega\odot dt+b_{2}(t)\cdot\mathcal{K}_{g_{Y}}(\omega),
\end{align} 
where each $b_{i}$ depends on $f$ and the eigenvalue of $\Delta_{H}$ on divergence-free 1-forms corresponding to $\eta$. Finally, on elements of type III we have
\begin{align}\label{imtypeIII}
\mathcal{D}^{*}\left(f(t)\cdot B \right)=\tilde{f}(t)\cdot B,
\end{align} 
where $\tilde{f}$ is determined by $f$ and the eigenvalue of $\Delta_{g_{Y}}$ corresponding to $B$.  
In a similar way to Corollary \ref{3types} one can prove that all elements in $S^{2}_{0}(T^{*}M)$ can be written uniquely as an infinite sum of elements as in the
right hand sides of \eqref{imtypeI}, \eqref{imtypeII} and \eqref{imtypeIII}, so it follows that in order to find the general solution of $\mathcal{D}^{*}Z=0$ it suffices to consider solutions $Z$ of types I, II and III separately. For example, if  $Z$ has the form \eqref{typeI} then writing $\mathcal{D}^{*}Z$ as in \eqref{imtypeI} one sees that in order to obtain $\mathcal{D}^{*}Z=0$ one must solve for $f$ in \eqref{typeI} so that in \eqref{imtypeII} one has  $a_{1}=a_{2}=a_{3}=a_{4}=0$ and in general this amounts to solving an ordinary differential equation on $f$.  We start by considering solutions of type III. For that purpose we use the following: 
\begin{lemma}\label{eigIII}
If $\lambda$ is an eigenvalue of $-\Delta_{g_{Y}}$
on divergence-free sections of $S^{2}_{0}(T^{*}Y)$, then 
\begin{enumerate}
\item[(a)] if $\kappa=1$,  $\lambda\ge 6$,
\item[(b)] if $\kappa=-1$, $\lambda\ge 3$ with equality achieved only for nontrivial Codazzi tensors $h \in S^{2}_{0}(T^{*}Y)$, that is, $d^{\nabla}h = 0$.  
\item[(c)] if $\kappa=0$, $\lambda\ge 0$ with equality for parallel sections in $S^{2}_{0}(T^{*}Y)$.
\end{enumerate}
\end{lemma}
\begin{proof} These are due to Koiso, we only give a brief argument
\cite{Koiso1}. For (a), the inequality
\begin{align}
\int_{S^3} | \nabla_{i} h_{jk} + \nabla_j h_{ki} + \nabla_{k} h_{ij}|^2 dV \geq 0,
\end{align}
easily implies that $\lambda\ge 6$. For (b), the inequality 
\begin{align}
\label{codin}
\int_Y  | \nabla_i h_{jk} - \nabla_j h_{ik} |^2 dV \geq 0,
\end{align}
implies that  $\lambda\ge 3$, with equality exactly for
Codazzi tensors.
Finally, the $\kappa = 0$ case is trivial. 
\end{proof}
The classification of type III solutions is given by the following.
\begin{proposition}\label{adjtype3}
Let $0\le \lambda_{1}<\lambda_{2}<\ldots$, be the eigenvalues of $-\Delta_{g_{Y}}$ on divergence-free tensors in $S^{2}_{0}(T^{*}Y)$ and let 
$\beta_{j}=\sqrt{\lambda_{j}+3\kappa}$. For each eigenvalue  $\lambda_{j}$ there exist trace-free and divergence-free eigentensors $B^{\pm}_{j}$ and $C^{\pm}_{j}$ 
satisfying
\begin{align}
\slashd B_j^{\pm} = \pm \beta_j B_j^{\pm}, \ \ \slashd C_j^{\pm} = \pm \beta_j C_j^{\pm},
\end{align}
such that the general solution of $\mathcal{D}^{*}Z=0$ with $Z$ satisfying
\begin{align}
\delta_{Y}Z=0,\label{divfree}\\
\tr_{Y}Z=0,\label{trfree}
\end{align}
can be written in in the following way:
\begin{itemize}
\item[(a)] If $\kappa=1$ then
\begin{align*}
Z=\sum_{j=1}^{\infty}\left(e^{(\beta_j + 1)t}B^{+}_{j} + e^{(\beta_j - 1)t}C^{+}_{j}
+ e^{(-\beta_j + 1)t}B^{-}_{j} + e^{(-\beta_j - 1)t}C^{-}_{j} \right).
\end{align*}
Letting $\alpha_{j}^{\pm}=\beta_{j}\pm 1$, we have 
$0<|\alpha^{\pm}_{1}|<|\alpha^{\pm}_{2}|<\ldots,$ and $|\alpha_{1}^{\pm}|=2$.
\item[(b)] If $\kappa=-1$
\begin{align*}
Z=\sum_{j=1}^{\infty} \left\{ e^{\beta_{j}t}\left(B^{+}_{j}\cos(t)+C^{+}_{j}\sin(t)\right)
+  e^{-\beta_{j}t}\left(B^{-}_{j}\cos(t)+C^{-}_{j}\sin(t)\right) \right\},
\end{align*}
 with $\beta_{1}=0$ and where $B^{\pm}_{1}$ and $C^{\pm}_1$ are trace-free Codazzi tensors.
\item[(c)] If $\kappa=0$,
\begin{align*}
Z= B_1 + t C_1 + \sum_{j=2}^{\infty} \left( e^{\beta_{j}t} B_j^+ +  t e^{\beta_{j}t} C_j^+ 
+ e^{-\beta_{j}t} B_j^- +  t e^{-\beta_{j}t} C_j^- \right),
\end{align*}
where $B_{1}$ and $C_1$ are parallel sections of $S^{2}_{0}(T^{*}Y)$.
\end{itemize}
\end{proposition}
\begin{proof}
Let $Z$ be a solution of $\mathcal{D}^{*}Z=0$ of type III, that is, $Z=fB$ with $B$ satisfying \eqref{divfree}, \eqref{trfree} and $\Delta_{g_{Y}}B=-\lambda\cdot B$, then $f$ and $B$ satisfy the equation
\begin{align*}
0=\mathcal{D}^{*}(fB)=\left\{ 0,0,-\frac{1}{2}\ddot{f}B-\frac{1}{2}\dot{f}\slashd B-\left(\kappa+\frac{\lambda}{2}\right)fB\right\},
\end{align*}
 from \eqref{slashdconstant} we have
\begin{align}\label{slashdfb}
\slashd{\dot{Z}}=\slashd\left(\dot{f}\cdot B\right)
=\pm \dot{f}\left( 2\sqrt{\lambda+3 \kappa}\cdot B\right).
\end{align}
It follows that $f$ is a solution of the ordinary differential equation
\begin{align}\label{odef}
- \frac{1}{2} \ddot{f}\pm\sqrt{\lambda+3\kappa}\cdot \dot{f} -\left(\kappa+\frac{\lambda}{2}\right) f =0.
\end{align}
Letting $\beta= \sqrt{\lambda+3\kappa}$, then the characteristic roots of \eqref{odef} are
\begin{align*}
\pm \beta\pm\sqrt{\kappa}.
\end{align*}
The expansions follow from considering the different solutions obtained for 
$\kappa \in \{1, -1, 0\}$, and Lemma \ref{eigIII}.
\end{proof}
We now turn to solutions of $\mathcal{D}^{*}(Z)=0$ with $Z$ of types  \textrm{I} and II. We will need to introduce the operator
\begin{align*}
\square_{\mathcal{K}}:\Lambda^{1}(T^{*}Y)\mapsto\Lambda^{1}(T^{*}Y),
\end{align*}
given by
\begin{align*}
\square_{\mathcal{K}}\eta=\delta_{Y}\mathcal{K}_{g_{Y}}(\eta).
\end{align*}
We have the following
\begin{lemma}\label{lemc}Let $\eta\in\Lambda^{1}(T^{*}Y)$, then
\begin{align}\label{boxK}
\square_{\mathcal{K}} \eta = (\delta_{Y} d + \frac{4}{3} d \delta_{Y} + 4\kappa) \eta. 
\end{align}
Also, if $\eta$ is an eigenform of $\Delta_{H}$ on 1-forms then 
\begin{align*}
\square_{\mathcal{K}}\eta=c\cdot\eta,
\end{align*}
where $c=(-\frac{4}{3}\mu+4\kappa)$ if $\eta=d\phi$ 
and $\Delta_{H}\phi=\mu\phi$ or $c=(4\kappa-\nu)$
if $\Delta_{H}\eta=\nu\eta$ and $\delta_{Y}\eta=0$. 
Moreover, in either case the constant $c$ is
nonzero unless $Z=0$. 
\end{lemma}
\begin{proof} The expression \eqref{boxK} for $\mathcal{K}_{g_{Y}}$ is a direct consequence of Lemma \ref{commutator}. Suppose now that $\Delta_{H}\eta=\lambda\cdot\eta$. If $\eta = d\phi$ with $\Delta_{H}\phi=\nu\phi$, then observe that
\begin{align*}
\square_{\mathcal{K}}\eta&=\left(\delta d+\frac{4}{3}d\delta+4\kappa\right)\eta\\
&=\left(\frac{4}{3}d\delta +4\kappa\right)\eta=\left(-\frac{4}{3}\mu+4\kappa\right)\eta.
\end{align*}
If $\delta_{Y} \eta = 0$, then  
\begin{align*}
\square_{\mathcal{K}}\eta&=\left(\delta_{Y}d+\frac{4}{3}d\delta_{Y}+4\kappa\right)\eta\\
&=\left(\delta_{Y}d+4\kappa\right)\eta
=-\Delta_{H}\eta+4\kappa\eta=\left(-\nu+4\kappa\right)\eta.
\end{align*}
Finally, in order to show that $c=0$ does not
occur we note that  when $\kappa=1$, there 
are eigenforms of $\Delta_{H}$ 
corresponding to the eigenvalue $\mu=3$ 
on  closed forms  and to $\nu=4$ on 
co-closed forms and in these cases $c=0$. 
However, for any of these eigenvalues, 
the corresponding eigenforms are
 conformally Killing.  In the hyperbolic case 
 $\kappa=-1$, the constant $c$ is strictly 
 negative for either closed or co-closed 
 eigenforms of the Hodge Laplacian. 
In the flat case $\kappa=0$, the 
constant $c$ equals zero only for parallel forms, 
but in this case $Z= 0$.
\end{proof}
Next, assume that $Z$ is a non-trivial solution of 
$\mathcal{D}^* Z = 0$ with $Z$ of type I or II and
$c\ne 0$, where $c$ is the constant in Lemma
 \ref{lemc}. The first component of \eqref{adjoint} 
yields 
\begin{align}
0= \delta_{Y} (\delta_{Y} Z) = f(t) \cdot \delta_{Y} \square_{\mathcal{K}} \eta
= f(t) \cdot \delta_{Y} ( c \eta). 
\end{align}
Since $Z$ is non-trivial and $c\ne 0$, we conclude 
that $\delta_{Y} \omega = 0$ and hence, solutions 
of type I do \emph{not} occur. Furthermore, we can prove
\begin{proposition}
We have 
\begin{align}\label{domega}
\dot{f} \tilde{*}d \omega = -\nu f  \omega. 
\end{align}
\end{proposition}
\begin{proof}
The second component of \eqref{adjoint} yields 
\begin{align}
0=\dot{f}\cdot\delta_{Y}\mathcal{K}_{g_{Y}}(\omega)-f\delta_{Y}\tilde{*}\delta_{Y}\mathcal{K}_{g_{Y}}(\omega),
\end{align}
which by Lemma \ref{commutator} is equivalent to
\begin{align}
\dot{f}(-\nu+4\kappa)\omega+f\delta_{H}\tilde{*}(-\nu+4\kappa)\omega=0
\end{align}
and writing $\delta_{H}$ on 2-forms as $\tilde{*}d\tilde{*}$ 
we obtain
\begin{align}
\label{dotf}
\dot{f}\omega=-f\tilde{*}d\omega,
\end{align}
and after taking $\tilde{*}d$ on both sides we conclude that
\begin{align}
\dot{f}\tilde{*}d\omega=-f\tilde{*}d\tilde{*}d\omega=-fd^{*}d\omega=-\nu f\omega,
\end{align}
as needed.
\end{proof}
\begin{proposition}\label{odetype2}
\label{kdprop} Let $Z = f(t) \cdot \mathcal{K}_{g_Y} \omega$ satisfy 
$\mathcal{D}^* Z = 0$,  
where $\omega$ satisfies $\delta_Y \omega = 0$ and 
\begin{align}
\Delta_{H} \omega &= \nu \cdot \omega.
\end{align}
Then 
\begin{align} 
f(t) = e^{\alpha t}, 
\end{align}
with
\begin{align}
\alpha = \pm ( \sqrt{ \nu} ),
\end{align}
or 
\begin{align}
\label{fttt}
f(t) = c_0,
\end{align}
for a constant $c_0$, if $\omega$ is a non-trivial harmonic $1$-form.
\end{proposition}
\begin{proof}
From \eqref{domega} we have 
\begin{align}\label{slashdZdot}
-\frac{1}{2}\slashd\dot{Z}&=-\frac{1}{2}\dot{f}\slashd\mathcal{K}_{g_{Y}}(\omega)=-\frac{1}{2}\dot{f}\mathcal{K}_{g_{Y}}(\tilde{*}d\omega)\nonumber\\
&=\frac{\nu}{2}f\mathcal{K}_{g_{Y}}(\omega),
\end{align}
we also have from Lemma \ref{commutator}
\begin{align}
\delta_{Y}Z&=f(t)\cdot\delta_{Y}\mathcal{K}_{g_Y}(\omega)\nonumber\\
&=f(t)\cdot(-\nu+4\kappa)\cdot\omega,\label{deltaZ}\\
\Delta_{g_{Y}}Z&=f(t)\cdot\Delta_{g_{Y}}\mathcal{K}_{g_{Y}}(\omega)\nonumber\\
&=f(t)\cdot (6\kappa-\nu)\cdot\mathcal{K}_{g_{Y}}(\omega),\label{laplacianZ}\\
\delta_{Y}^{2}Z&=f(t)\cdot\delta^{2}_{Y}\mathcal{K}_{g_Y}(\omega)\nonumber\\
&=f(t)\cdot(-\nu+4\kappa)\cdot \delta_{Y}\omega=0.\label{deltasquareZ}
\end{align}
The equation on the purely spherical component of $\mathcal{D}^{*}(Z)$ is
\begin{align}
0=- \frac{1}{2} \ddot{Z} -\kappa Z - \frac{1}{2} \slashd \dot{Z}
+ \frac{1}{2} \Delta_{g_{Y}} Z
 - \frac{1}{2} \mathcal{L}_{g_Y}(\delta_{Y} Z) +\frac{1}{2}(\delta_{Y}^{2}Z)g_{Y},
\end{align}
which by \eqref{deltaZ}, \eqref{laplacianZ},\eqref{deltasquareZ} and \eqref{slashdZdot} simplifies to
\begin{align}
0&=\left(-\frac{1}{2}\ddot{f}-\kappa f+\frac{1}{2}(6\kappa-\nu)f-\frac{1}{2}(4\kappa-\nu)\right)\mathcal{K}_{g_{Y}}(\omega)-\frac{1}{2}\dot{f}\mathcal{K}_{g_{Y}}(\tilde{*}d\omega)\nonumber\\
&=-\frac{1}{2}\ddot{f}\mathcal{K}_{g_{Y}}(\omega)+\frac{1}{2}\nu f\mathcal{K}_{g_{Y}}(\omega),
\end{align}
which we write as
\begin{align}
\left(\ddot{f}-\nu\right)\mathcal{K}_{g_{Y}}(\omega)=0.
\end{align}
and for 1-forms $\omega$ that are \emph{not} dual to Killing fields we obtain solutions
\begin{align}
f(t)=e^{\pm\sqrt{\nu}t}.
\end{align}
The solutions with $f$ as in \eqref{fttt} correspond to $\nu = 0$ 
which is the case of a harmonic $1$-form. In this case, 
the tensor $Z = t \cdot \mathcal{K}_{g_Y} (\omega)$ is ruled out by \eqref{dotf} above, and 
only the solution $Z = c_0 \cdot \mathcal{K}_{g_Y}( \omega)$ occurs. 
\end{proof}

Let $0=\nu_{0}<\nu_{1}<\ldots,$ be all the eigenvalues of $\Delta_{H}$ on co-closed forms in $\Lambda^{1}(T^{*}Y)$. Note that there are non-trivial eigenforms corresponding to $\nu_{0}$ if and only if $b_{1}(T^{*}Y)\ne 0$ since such eigenforms are harmonic 1-forms. In particular, for $\kappa=1$ there are no nontrivial 1-forms in $\Lambda^{1}(T^{*}Y)$. We close this section with the following.
\begin{proposition}
\label{sec5prop}
Let $Z\in S^{2}_{0}(T^{*}Y)$ be a solution of $\mathcal{D}^{*}Z=0$. Then $Z$ can be written as 
\begin{align*}
Z=\mathcal{K}_{g_{Y}}(\omega)+Z_{0},
\end{align*}
where $Z_{0}$ is divergence-free and has an expansion as in Proposition \ref{adjtype3}. Also for each eigenvalue $\nu_{j}$, $j=0,1,\ldots$, there are eigenforms $\omega^{\pm}_{j}$ such that $\mathcal{K}_{g_{Y}}(\omega)$ can be written uniquely as an infinite sum of the following form
\begin{enumerate}
\item[(a)] If $\kappa=1$ then 
\begin{align*}
\mathcal{K}_{g_Y}(\omega)=\sum_{j=2}^{\infty}(e^{\sqrt{\nu_{j}}t}\mathcal{K}_{g_Y}(\omega^{+}_{j})+e^{-\sqrt{\nu_{j}}t}\mathcal{K}_{g_{Y}}(\omega^{-}_{j})).
\end{align*}
where $c^{\pm}_{j}$ are constants. In the case $Y = S^3$, $\nu_{j}=(j+1)^{2}$.
\item[(b)] If $\kappa=-1$ then 
\begin{align*}
\mathcal{K}_{g_Y}(\omega)= \mathcal{K}_{g_Y}(\omega_{0})
 +\sum_{j=1}^{\infty}\left(e^{\sqrt{\nu_{j}}t}\mathcal{K}_{g_{Y}}(\omega^{+}_{j})+e^{-\sqrt{\nu_{j}}t}\mathcal{K}_{g_{Y}}(\omega^{-}_{j})\right),
\end{align*} 
where $\omega_{0}$ is a harmonic $1$-form. 
\item[(c)] If $\kappa=0$ then 
\begin{align*}
\mathcal{K}_{g_Y}(\omega)= \sum_{j=1}^{\infty}\left(e^{\sqrt{\nu_{j}}t}\mathcal{K}_{g_{Y}}(\omega^{+}_{j})+e^{-\sqrt{\nu_{j}}t}\mathcal{K}_{g_{Y}}(\omega^{-}_{j})\right).
\end{align*} 
\end{enumerate}
\end{proposition}
\begin{proof}
From Proposition \ref{odetype2}, we can write the 1-form $\omega$ as an infinite sum of the form
\begin{align}
\label{omegaexp}
\omega=\omega_{0}+\sum_{j=1}^{\infty}\left(e^{\sqrt{\nu_{j}}t}\omega^{+}_{j}+c_{j}^{-}e^{-\sqrt{\nu_{j}}t}\omega^{-}_{j}\right). 
\end{align}
In case $\kappa = 1$, there are no harmonic $1$-forms, and all eigenforms corresponding to $\nu_{1}=4$ are dual to Killing fields, so the sum starts at
$j = 2$ in this case.
The form of the eigenvalues $\nu_{j}$ in the case $\kappa=1$ follows from \cite{foll}.
The $\kappa = -1$ case follows directly from \eqref{omegaexp}. In the case
$\kappa = 0$, any harmonic $1$-form is parallel. 
\end{proof}

%%%%%%%%%%%%%%%%%%%%%%%%%%%%%%%%%%%%%%%%%%%%%%%%
\section{Mixed solutions}
\label{mixedsec}
%%%%%%%%%%%%%%%%%%%%%%%%%%%%%%%%%%%%%%%%%%%%%%%%%%%%%%%%%%%%
Returning to the full system 
\begin{align}\label{mix}
\mathcal{D}^{*}Z=\mathcal{K}_{g}(\tilde{\omega}),
\end{align}
we note that since $\mathcal{D}^{*}Z$ is divergence-free, the 1-form $\tilde{\omega}$ automatically satisfies the equation
\begin{align}\label{divconfkill}
\delta_{g}\mathcal{K}_{g}(\tilde{\omega}) \equiv \bkill(\tilde{\omega}) = 0,
\end{align}
so we next analyze solutions of \eqref{divconfkill}
at a cylindrical metric $g=dt^2+g_{Y}$. The conformal Killing operator
on a 1-form $\tilde{\omega}=fdt+\omega$ is
\begin{align}
\label{kform}
\mathcal{K}_{g}(\tilde{\omega})=\left(\frac{3}{2}\dot{f}-\frac{1}{2}\delta_{Y}\omega\right)dt\otimes dt+\left(\dot{\omega}+df\right)\odot dt+\mathcal{L}_{g_Y}(\omega)-\frac{1}{2}\left(\dot{f}+\delta_{Y}\omega\right)g_{Y}.
\end{align}
The divergence of a traceless symmetric $2$-tensor 
\begin{align}
\tilde{h} = h_{00} dt\otimes dt + ( \alpha \otimes dt + dt \otimes \alpha) 
+ h
\end{align}
is given by
\begin{align}
\label{hform}
\delta \tilde{h} 
=  ( \dot{h}_{00} + \delta_{S^3} \alpha) dt +  \dot{\alpha}  + \delta_{S^3} h.
\end{align}
Combining \eqref{kform} and \eqref{hform}, we obtain
\begin{align*}
\bkill\tilde{\omega}&=\left(\frac{3}{2}\ddot{f}+\frac{1}{2}\delta_{Y}\dot{\omega}-\Delta_{H}f\right)dt+\ddot{\omega}+\delta_{Y}\mathcal{L}_{g_Y}(\omega)-\frac{1}{2}d\delta_{Y}\omega.
\end{align*}
Commuting covariant derivatives as in Lemma \ref{commutator}, we have
\begin{align*}
\delta_{Y}\mathcal{L}_{g_Y}(\omega)=-\Delta_{H}\omega+d\delta_{Y}\omega+4\kappa\omega,
\end{align*}
so $\bkill\tilde{\omega}$ takes the form 
\begin{align*}
\bkill(\tilde{\omega})=\left(\frac{3}{2}\ddot{f}+\frac{1}{2}\delta_{Y}\dot{\omega}-\Delta_{H}f\right)dt+\ddot{\omega}-\Delta_{H}\omega+\frac{1}{2}d\delta_{Y}\omega+4\kappa\omega+\frac{1}{2}d\dot{f}.
\end{align*}
Any 1-form $\tilde{\omega}\in\Lambda^{1}(T^{*}M)$ can be written as an infinite sum of 1-forms of two types, namely
\begin{enumerate}
\item[(i)]  Forms of type (a)
\begin{align}\label{typea}
c(t)\phi dt+k(t)d\phi,
\end{align}
where $\phi$ is an eigenfunction of the Hodge laplacian $\tilde{\Delta}_{H}$ on $\Lambda^{0}(T^{*}Y)$ with eigenvalue $\mu$ and $c=c(t)$, $k=k(t)$ are functions of $t$,
\item[(ii)] Forms of type (b)
\begin{align}\label{typeb}
m(t)\eta,
\end{align} 
where $\eta$ is an eigenform of the Hodge Laplacian $\Delta_{H}$ on $\Lambda^{1}(T^{*}Y)$ satisfying
\begin{align*}
\delta_{Y}\eta&=0,
\end{align*}
and $m=m(t)$. 
\end{enumerate}
Let us start by solving $\bkill\tilde{\omega}=0$ assuming that $\tilde{\omega}$ is of type (a).  In this case, from \eqref{typea} we conclude that the functions $l$ and $m$ satisfy the following system of ordinary differential equations
\begin{align}\label{bkilltypea}
\begin{split}
\ddot{l}&=\frac{2}{3}\mu l+\frac{\mu}{3}\dot{m},\\
\ddot{m}&=-\frac{1}{2}\dot{l}+\left(\frac{3}{2}\mu-4\kappa\right)m.
\end{split}
\end{align}
If we let $l_{1}=l$, $l_{2}=\dot{l}_{1}$, $m_{1}=m$ and $m_{2}=\dot{m}_{1}$, the system \eqref{bkilltypea} is equivalent to the first order linear system
\begin{align*}
\dot{X}=AX,
\end{align*}
where $X$ and $A$ are given by
\begin{align*}
X=\left(
\begin{array}{c}
l_{1}\\
l_{2}\\
m_{1}\\
m_{2}
\end{array}
\right),~
A=\left(
\begin{array}{cccc}
0&1&0&0\\
\frac{2}{3}\mu&0&0&\frac{\mu}{3}\\
0&0&0&1\\
0&-\frac{1}{2}&\left(\frac{3}{2}\mu-4\kappa\right)&0 
\end{array}
\right).
\end{align*}
The characteristic roots of the matrix $A$ are $\pm\alpha^{\pm}(\mu)$ where $\alpha^{\pm}(\mu)$ is given by
\begin{align}\label{crtypea}
\alpha^{\pm}=\alpha^{\pm}(\mu)=\sqrt{\mu-2\kappa\pm 2\sqrt{\kappa^{2}-\frac{\mu}{3}\kappa}},
\end{align}
We now consider solutions of $\square_{\mathcal{K}}\tilde{\omega}=0$ with $\tilde{\omega}$ of type (b).  If $\tilde{\omega}$ is as in \eqref{typeb},  the system $\square_{\mathcal{K}}\tilde{\omega}=0$ takes the form
\begin{align}\label{crtypeb}
\ddot{m}-\nu m+4\kappa m=0,
\end{align}
and the characteristic roots of this equation are
\begin{align*} 
\pm\sqrt{\nu-4\kappa}.
\end{align*}

Let $0=\mu_{0}<\mu_{1}<\ldots$ be all the eigenvalues of $\Delta_{H}$ on $\Lambda^{0}(T^{*}Y)$ and let $\nu_{j}$ for $j=0,1,\ldots,$ denote all the eigenvalues of $\Delta_{H}$ on co-closed forms in $\Lambda^{1}(T^{*}Y)$. In particular if $\kappa=1$
and $\Gamma = \{e\}$, $\mu_{j}=j(j+2)$. We have
 
\begin{proposition}
Let $(Z, \tilde{\omega})$ be a solution of \eqref{mix}.  Then up to addition of 1-forms which are dual to conformal Killing fields, the 1-form $\tilde{\omega}$ can be written as follows.
\begin{enumerate}
\item[(a)] If $\kappa=1$ and $\Gamma = \{e\}$,  
$\tilde{\omega}$ is an infinite sum of the form
\begin{align*}
&\sum_{j=2}^{\infty}e^{\pm\beta_{j}t}\left(\left\{\phi_{1j}^{\pm}\cos(\gamma_{j}t)+\phi_{2j}^{\pm}\sin(\gamma_{j}t)\right\}dt+\left\{c^{\pm}_{1j}\cos(\gamma_{j}t)d\phi^{\pm}_{1j}+c_{2j}^{\pm}\sin(\gamma_{j}t)d\phi^{\pm}_{2j}\right\}\right)\\
&+\sum_{j=2}^{\infty}\left(e^{\delta_j t}\omega^{+}_{j}+e^{-\delta_j t}\omega^{-}_{j}\right),
\end{align*} 
where $\phi^{\pm}_{1j}$ and $\phi^{\pm}_{2j}$ are eigenfunctions of $\Delta_{H}$ corresponding to $\mu_{j}=j(j+2)$, and the coefficients $c_{1j}^{\pm}$ and $c^{\pm}_{2j}$ are constants,
for $j \geq 2$. The rates $\beta_{j}$ satisfy $\sqrt{6}<\beta_{j}$ for $j \geq 2$ 
and are given by $\beta_j = Re( \alpha_j^+)$, $\gamma_{j} = Im( \alpha_j^+)$, where 
\begin{align*}
\alpha_{j}^{\pm}= \sqrt{j(j+2)-2 \pm \frac{2}{3}\sqrt{-1} \cdot \sqrt{3(j-1)(j+3)}}.
\end{align*}
Also, $\omega^{\pm}_{j}$ are eigenforms corresponding to the eigenvalues 
$\nu_{j} = (j+1)^2$ of $\Delta_{H}$ on co-closed forms, and 
$\delta_j = \sqrt{ \nu_j -4}$. 

If $\Gamma \neq \{e\}$ then $\pm \alpha_j^{\pm}$ or $\pm \delta_j$ 
will occur as indicial roots if and only if the corresponding eigenfunction or 
eigenform descends to the quotient $S^3/\Gamma$, respectively. 

\item[(b)] If $\kappa=-1$, then 
\begin{align*}
\tilde{\omega}=&\sum_{j=1}^{\infty}\Big( \left\{e^{\pm\sigma^{+}_{j}t}\phi^{\pm}_{1j}+e^{\pm\sigma^{-}_{j}t}\phi^{\pm}_{2j}\right\}dt+c^{\pm}_{1j}e^{\pm\sigma_{j}^{+}t}d\phi^{\pm}_{1j}+c^{\pm}_{2j}e^{\pm\sigma^{-}_{j}t}d\phi^{\pm}_{2j} \Big)\\
&\sum_{j=0}^{\infty}\left(e^{\tau_{j}t}\omega^{+}_{j}+e^{-\tau_{j}t}\omega^{-}_{j}\right),
\end{align*}
where $\omega_{0}^{\pm}$ are harmonic 1-forms in $\Lambda^{1}(T^{*}Y)$.  The numbers $\sigma^{\pm}_{j}$ for $j \geq 1$ are real and are given by 
\begin{align}
\sigma^{\pm}_{j}=\sqrt{\mu_{j}+2\pm2\sqrt{1+\frac{\mu_{j}}{3}}},
\end{align}
where $\mu_j$ are the eigenvalues with respect to the hyperbolic metric. 
The numbers $\tau_{j}$ are also real and are given by $\tau_{j}=\sqrt{\nu_{j}+4}$,
where $\nu_j$ are the eigenvalues with respect to the hyperbolic metric. 
The coefficients $c_{1j}^{\pm}$ and $c^{\pm}_{2j}$ for $j \geq 1$ are constants.
\item[(c)] If $\kappa=0$, 
\begin{align*}
\tilde{\omega}=\sum_{j=1}^{\infty}e^{\pm t\sqrt{\mu_{j}}}\left(\left\{\phi_{1j}^{\pm}+ t\phi_{2j}^{\pm}\right\}dt+\left\{c_{1j}^{\pm} d\phi_{1j}^{\pm} + tc_{2j}^{\pm} d\phi_{2j}^{\pm} \right\}\right)
+\sum_{j=1}^{\infty}\left(e^{t\sqrt{\nu_{j}}} \omega_j^+ +e^{-t\sqrt{\nu_{j}}} \omega_j^-\right),
\end{align*}
where the notation is as above, but with eigenvalues 
$\mu_j$ and $\nu_j$ corresponding to the metric on $T^3$.  
\end{enumerate}
\end{proposition}
\begin{proof}
For the case $\kappa=1$, the only real roots in \eqref{crtypea} correspond to the eigenvalues $\mu=0,3$ of $\Delta_{H}$ on $\Lambda^{0}(T^{*}Y)$, however, we see in either case that for the solution $\tilde{\omega}$ of \eqref{divconfkill} obtained, $\mathcal{K}_{g}(\tilde{\omega})$ is not in the image of $\mathcal{D}^{*}$. To clarify this observation, we note that for $\mu=0$,  $\mathcal{K}_{g}(\tilde{\omega})$ has the form
\begin{align}\label{muzero}
l(t)dt \otimes dt+k(t)g_{Y},
\end{align}
i.e., $l$ and $k$ only depend on $t$ and for $\mu=3$, $\mathcal{K}_{g}(\tilde{\omega})$ has the form
\begin{align}\label{mu3}
 c_{1}(t)\phi dt\otimes dt+c_{2}(t)d\phi\odot dt+c_{3}(t)\phi g_{Y},
 \end{align}
where $\phi$ is a spherical harmonic of order 1 (and hence $d\phi$ is conformally Killing with respect to the metric $g_{Y}$). From \eqref{imtypeI}, \eqref{imtypeII} and \eqref{imtypeIII},  we conclude that elements of the form \eqref{muzero} or \eqref{mu3} in the image of $\mathcal{D}^{*}$ can only arise from evaluating $\mathcal{D}^{*}$ at an element of the form $f(t)\cdot\mathcal{K}_{g_{Y}}(d\psi)$, but in either case $\mathcal{K}_{g_{Y}}(d\psi)=0$ and therefore $\tilde{\omega}$ must be conformally Killing with respect to the metric $g$. Similarly, for forms of type (b) in the case $\kappa=1$, we see that the solution of \eqref{divconfkill} obtained for the least positive eigenvalue of $\Delta_{H}$ on co-closed forms in $\Lambda^{1}(T^{*}Y)$ (which is $\nu=4$), is dual to a Killing field in $Y$, and in this case $\mathcal{K}_{g}(\tilde{\omega})$ is not in the image of $\mathcal{D}^{*}$ either. It is also easy to see from \eqref{crtypea} that for $\mu>3$ all the rates $\alpha^{\pm}(\mu)$ satisfy $|Re(\alpha^{\pm}(\mu))|>\sqrt{6}$. The rest of the Proposition follows from straightforward computations. 

\end{proof}

We are now ready to describe the general solution of \eqref{mix}. If $(Z,\tilde{\omega})$ is a solution of \eqref{mix}, then we can write $Z$ as
\begin{align*}
Z=Z_{0}+\tilde{Z},
\end{align*}
where $Z_{0}$ satisfies $\mathcal{D}^{*}Z_{0}=0$ and $\tilde{Z}$ is a non-zero solution of  \eqref{mix}. We now prove that this solution $\tilde{Z}$ indeed exists.

\begin{proposition}
\label{sec6prop}
Let $\tilde{\omega}\in\Lambda^{1}(T^{*}M)$ be a solution  of \eqref{divconfkill}
of type (a) or type (b)  which is not conformally Killing. 
There exists a nonzero $\tilde{Z}\in S^{2}_{0}(T^{*}M)$ such that 
$\mathcal{D}^{*}\tilde{Z}=\mathcal{K}_g\tilde{\omega}$.
\end{proposition}

\begin{proof}
For the proof, we consider a solution of \eqref{mix} with $\tilde{\omega}$ of type 
(a), that is, 
\begin{align*}
\tilde{\omega}=l\phi dt+md\phi,
\end{align*}
where $\phi$ is an eigenfunction of $\Delta_{H}$ on $\Lambda^{0}(T^{*}Y)$ with eigenvalue $\mu$ and $l,m$ are solutions of \eqref{bkilltypea}. On the other hand, for the element of type I, $f(t)\cdot\mathcal{K}_{g_{Y}}(d\phi)$, the operator $\mathcal{D}^{*}$ can be computed explicitly following the results in Section \ref{adjointeq} as
\begin{align*}
\mathcal{D}^{*}\left(f\mathcal{K}_{g_{Y}}(d\phi)\right)=\Big\{\mu(2\kappa-\frac{2}{3}\mu)f\phi,\dot{f}(2\kappa-\frac{2}{3}\mu)d\phi,\\
-\frac{1}{2}\left(\ddot{f}-\frac{\mu}{3}f\right)\mathcal{L}_{g_Y}(d\phi)-\frac{\mu}{3}\left(\ddot{f}-(\mu-2\kappa)f\right)\phi g_{Y}\Big\}.
\end{align*}
Suppose that $d\phi$ is not conformally Killing with respect to $g_{Y}$, then 
$\mathcal{L}_{g_Y}(d\phi)$ and $\phi g_{Y}$ are linearly independent. 
If we write $Z=f\mathcal{K}_{g_{Y}}(d\phi)$, then we can solve for $f$ such that $\mathcal{D}^{*}Z=\mathcal{K}_{g}(\tilde{\omega})$ by considering the system 
\begin{align}\label{flm}
\begin{split}
\mu(2\kappa-\frac{2}{3}\mu)f&=\frac{3}{2}\dot{l}+\frac{\mu}{2}m,\\
(2\kappa-\frac{2}{3}\mu)\dot{f}&=l+\dot{m},\\
-\frac{1}{2}\left(\ddot{f}-\frac{\mu}{3}f\right)&=m,\\
-\frac{\mu}{3}\left(\ddot{f}-(\mu-2\kappa)f\right)&=-\frac{1}{2}(\dot{l}-\mu m).
\end{split}
\end{align} 
Note that from the condition that $d\phi$ is not 
conformally Killing we see that in the cases $\kappa=\pm 1$ we 
have $\mu(2\kappa-\frac{2}{3}\mu)\ne 0$ and
then, from \eqref{bkilltypea} we see that if we set
\begin{align*}
f=\frac{1}{\mu(2\kappa-\frac{2}{3}\mu)}\left(\frac{3}{2}\dot{l}+\frac{\mu}{2}m\right),
\end{align*}
then $f$  is a a nontrivial solution of 
the system \eqref{flm} unless $\tilde{\omega}$ 
is conformally Killing with respect to the 
metric $g=dt^2+g_{Y}$  and 
hence $Z=f\mathcal{K}_{g_{Y}}(d\phi)$ 
is a nontrivial solution of \eqref{mix}. The case $\kappa = 0$ is similar. 

If now $\tilde{\omega}$ is of type (b),
we can write
$\tilde{\omega}=m\eta$
where $\eta\in\Lambda^{1}(T^{*}Y)$ 
is a co-closed eigenform of
$\Delta_{H}$ with eigenvalue 
$\nu$. Let us also consider
an element of type II written as
$Z=f\mathcal{K}_{g_{Y}}(\eta)$ where $f$ is a function. 
 Assuming that $\eta$ 
 is not conformally Killing,
 we have
\begin{align*}
\mathcal{D}^{*}Z=\left\{0,\frac{1}{2}(4\kappa-\nu)\left(\dot{f}\pm\sqrt{\nu}f\right)\eta,\frac{1}{2}\left(-\ddot{f}\mp \sqrt{\nu}\dot{f}\right)\mathcal{L}_{g_Y}(\eta)\right\},
\end{align*}
where the sign of $\pm\sqrt{\nu}$ arises from Corollary \ref{stard}. 
In order to 
solve
 $\mathcal{D}^{*}Z=\mathcal{K}_{g}(\tilde{\omega})$, 
 we consider $\tilde{\omega}=m\eta$, where 
 $m$ solves \eqref{crtypeb} so \eqref{mix} reduces to the system
\begin{align}\label{fm}
\begin{split}
\frac{1}{2}(4\kappa-\nu)\left(\dot{f}\pm\sqrt{\nu}f\right)&=\dot{m},\\
-\frac{1}{2}\left(\ddot{f}\pm \sqrt{\nu}\dot{f}\right)=m.
\end{split}
\end{align}  
and again since $\eta$ is not conformally Killing with respect to $g_{Y}$  it follows that $4\kappa-\nu$ is non-zero and if we find $f$ satisfying
\begin{align*}
\dot{f}\pm\sqrt{\nu}f=\frac{2\dot{m}}{4\kappa-\nu}, 
\end{align*}
 then $f$ is a solution of \eqref{fm} and $f\mathcal{K}_{g_{Y}}(\eta)$
is a solution of \eqref{mix} with $\tilde{\omega}=m\eta$.  We can choose $f$ to be
 \begin{align*}
 f(t)=e^{\pm\sqrt{\nu}t}f_{0}+\frac{2e^{\pm\sqrt{\nu}t}}{4\kappa-\nu}\int_{0}^{t}\dot{m}(s)e^{\mp\sqrt{\nu}s}ds,
 \end{align*}
 where $f_{0}$ is a constant. It is clear that we can 
 choose the constant $f_{0}$
 so that $f$ is a solution of
  \eqref{crtypeb}. The case $\kappa = 0$ is similar. 
\end{proof}

%%%%%%%%%%%%%%%%%%%%%%%%%%%%%%%%%%
\section{Completion of proofs}
\label{proofsec}
%%%%%%%%%%%%%%%%%%%%%%%%%%%%%%%%%%%%%%%%%%%%%

We first state the following which determines all indicial 
roots of $F^*$ in the spherical case:
\begin{theorem}
\label{s3thm2}
Let $M$ be  $\R \times S^{3}/ \Gamma$ with 
product metric $g = dt^2 + g_{S^3/\Gamma}$, where $g_{S^3/\Gamma}$ is a 
metric of constant curvature $1$.  
Let $\mathcal{I}^*$ denote the set of indicial roots  of $F^*$. 

\begin{itemize} 
\item Case (0): $0 \in \mathcal{I}^*$.  

\item Case (1): If $\Gamma = \{e\}$ then $j = \pm 1 \in \mathcal{I}^*$. 
If $\Gamma$ is non-trivial, then  $j = \pm 1 \notin \mathcal{I}^*$.

\end{itemize}
All solutions in Case (0) and Case (1) are of the form $(0,\omega)$,
where $\omega$ is dual to a conformal Killing field 
(that is, $\mathcal{K}_g \omega = 0$). 

\begin{itemize}

\item Case (2): If $B$ is a nontrivial eigentensor of 
$\Delta_{S^3/\Gamma}$ on divergence-free symmetric $2$-tensors,
with eigenvalue $j^2 + 2j -2$ with $j \geq 2$, 
then $\{\pm j, \pm(j+2) \} \in \mathcal{I}^*$. 

\item Case (3): 
If $\omega$ is an eigenform of $\Delta_{S^3/\Gamma}$ 
on divergence-free $1$-forms with eigenvalue $(j+1)^2$, with $j \geq 2$,
then $\pm (j+1) \in \mathcal{I}^*$. 

\end{itemize}
All solutions in Case (2) and Case (3) are of the form $(Z,0)$.

\begin{itemize} 

\item Case (4) If $u$ is an eigenfunction of $\Delta_{S^3/\Gamma}$
with eigenvalue $j (j+2), j \geq 2$ then 
\begin{align}
\label{alphajdef}
\pm \alpha_j^{\pm} = \pm \sqrt{ j (j+2) -2  \pm \frac{2}{9} \sqrt{-1} \cdot \sqrt{(j+3)(j-1)} } \in \mathcal{I}^* 
\end{align}

\item Case (5) If $\omega$ is an eigenform of $\Delta_{S^3/\Gamma}$ 
on divergence-free $1$-forms with eigenvalue $(j+1)^2$, with $j \geq 2$,
then 
\begin{align}
\label{deltajdef}
\pm \delta_j = \pm \sqrt{ (j + 1)^2 - 4} \in \mathcal{I}^*. 
\end{align}
\end{itemize}
All solutions in Case (4) and Case (5) are of the form $(Z,\omega)$
with both $Z$ and $\omega$ nontrivial and $\mathcal{K}_g \omega \neq 0$. 

\end{theorem}

\begin{remark}{\em 
If $\Gamma = \{e\}$, then all of the above indicial 
roots do in fact occur. For nontrivial $\Gamma$, exactly which 
roots occur depends on which eigentensors descend from $S^3$ to the 
quotient $S^3/\Gamma$. 
}
\end{remark}
\begin{proof}[Proof of Theorem \ref{s3thm2}]
This follows from combining Propositions \ref{sec5prop} 
and \ref{sec6prop} for the case $\kappa = 1$. 
\end{proof}

\begin{proof}[Proof of Theorem \ref{s3thm}]
This follows immediately from Theorem \ref{s3thm2},
since Cases (2) and (3) obviously have real part 
larger than $2$, and it is easy to see that 
$|Re(\alpha_j^{\pm})| > \sqrt{6}$ and $|Re(\delta_j)| \geq \sqrt{5}$ for all 
$j \geq 2$. The determination of the conformal Killing fields 
follows easily from Section \ref{mixedsec}. 
\end{proof}

Next we state the following Theorem, which immediately 
implies Theorem \ref{kershort}. 

\begin{theorem}
\label{kerlong}
Let $M$ be  $\R \times S^{3}/ \Gamma$ with product metric $dt^2 + g_{S^3/\Gamma}$,
where $g_{S^3/\Gamma}$ is a metric of constant curvature $1$.  
Let $\mathcal{I}$ denote the set of indicial roots  of $F$. 

\begin{itemize} 
\item Case (0): $0 \in \mathcal{I}$.  
The corresponding kernel is 
\begin{align}
\label{case0}
span \{  3 dt \otimes dt - g_{S^3}, dt \odot \omega_0  \}
\end{align}
where $\omega_0$ is dual to a Killing field on $S^3/ \Gamma$.

\item Case (1): If $\Gamma = \{e\}$ then $j = \pm 1 \in \mathcal{I}$. 
If $\Gamma$ is non-trivial, then  $j = \pm 1 \notin \mathcal{I}$.
The corresponding kernel elements are given by 
\begin{align}
\label{case1}
h_{\phi} = p(t) \phi ( 3 dt \otimes dt - g_{S^3}) + q(t) (dt \odot d\phi), 
\end{align}
where $p(t) = C_3 e^t - C_4 e^{-t}$ and $q(t) = C_3 e^t + C_4 e^{-t}$, 
for some constants $C_3$ and $C_4$, and $\phi$ is 
a lowest nonconstant eigenfunction of $\Delta_{S^3/\Gamma}$.
In particular, if $\Gamma$ is nontrivial, 
then $j = \pm 1$ are not indicial roots.
\end{itemize}
All solutions in Case (0) and Case (1) are in the image of the 
conformal Killing operator. More precisely, 
\begin{align}
\label{case01}
 3 dt \otimes dt - g_{S^3} = \mathcal{K}_{g}(2tdt) \mbox{ and }  
dt \odot \omega_0 = \mathcal{K}_{g}(t \omega_0),
\end{align}
and 
\begin{align}
\label{case01'}
h_{\phi} = \mathcal{K}_{g} \Big\{  \frac{1}{2} \left( 
C_3(t+3)e^t - C_4(t-3)e^{-t} \right)  \phi dt
+  \frac{1}{2}\left(  - C_3te^t - C_4 te^{-t} \right) d\phi \Big\}.
\end{align}

\begin{itemize}

\item Case (2): If $B$ is a nontrivial eigentensor of 
$\Delta_{S^3/\Gamma}$ on divergence-free symmetric $2$-tensors,
with eigenvalue $j^2 + 2j -2$ with $j \geq 2$, 
then $\{\pm j, \pm(j+2) \} \in \mathcal{I}$. 

\item Case (3): 
If $\omega$ is an eigenform of $\Delta_{S^3/\Gamma}$ 
on divergence-free $1$-forms with eigenvalue $(j+1)^2$, with $j \geq 2$,
then $\pm (j+1) \in \mathcal{I}$. 

\end{itemize}
The kernel elements in Case (2) are of the form $h = f(t) B$, and
in Case (3) are of the form $h = f_{0}(t)\cdot \omega\odot dt+f_{1}(t)\cdot \mathcal{K}_{S^3}(\omega)$. 
Neither of these are in the image of the conformal 
Killing operator $\mathcal{K}_{g}$ of the cylinder. 
\begin{itemize} 

\item Case (4) If $u$ is an eigenfunction of $\Delta_{S^3/\Gamma}$
with eigenvalue $j (j+2), j \geq 2$ then $\pm \alpha_j^{\pm} \in \mathcal{I}$,
where $\alpha_j^{\pm}$ were defined in \eqref{alphajdef}. 

\item Case (5) If $\omega$ is an eigenform of $\Delta_{S^3/\Gamma}$ 
on divergence-free $1$-forms with eigenvalue $(j+1)^2$, with $j \geq 2$,
then $\pm \delta_j \in \mathcal{I}$, where $\delta_j$ were defined in \eqref{deltajdef}. 
\end{itemize}
All solutions in Case (4) and Case (5) are in the image of the 
conformal Killing operator of the cylinder. More precisely, they 
are exactly those solutions of $\square_{\mathcal{K},g} \omega = 0$ which are
not conformally Killing. 
\end{theorem}

\begin{remark}{\em 
As before, if $\Gamma = \{e\}$, then all of the above indicial 
roots do in fact occur. For nontrivial $\Gamma$, exactly which 
roots occur depends on which eigentensors descend to the 
quotient.
}
\end{remark}

\begin{proof}[Proof of Theorem \ref{kerlong}]
From the index theorem of Lockhart-McOwen, it follows that the 
real parts of indicial roots of $F$ are the same as those of $F^*$ and 
the dimensions of the space of solutions of the form 
$e^{\lambda t}p(y,t)$  where $p$ is a polynomial in $t$ with 
coefficients in $C^{\infty}(Y)$ are the same for all 
indicial roots with the same real part. We consider Cases (0) -- (5) in order. 

For Case (0), the corresponding kernel of $F^*$ is of the 
form $(0, \omega)$, where $\omega$ is dual to a 
bounded conformal Killing field on the cylinder. 
By direct calculation, elements in \eqref{case0} 
form the corresponding space of kernel elements. 

For Case (1), the corresponding  kernel of $F^*$ is of the 
form $(0, \omega)$, where $\omega$ is dual to a 
conformal Killing field which grows like $e^t$ on one end. 
For $S^3$, this is an $8$-dimensional space, while 
if $\Gamma$ is nontrivial, this space is empty. 
Again, by direct calculation, elements in \eqref{case1} 
form the corresponding $8$-dimensional space of kernel elements
in the case of the sphere. 
The formulas in \eqref{case01} and \eqref{case01'} can also 
easily be verified by direct calculation, which we omit. 

For Case (2), we consider solutions of the form
\begin{align*}
f(t)\cdot B,
\end{align*} 
where $f$ is a function, and $B\in S^{2}_{0}(T^{*}Y)$ is an eigentensor of $\Delta_{g_{Y}}$ with eigenvalue $-\lambda$  satisfying $\tr_{Y}(B)=0$ and $\delta_{Y}B=0$. 
In this case the equation $F=0$ takes the form
\begin{align*}
-\frac{1}{2} \ddot{f}\cdot B-f\cdot B-\frac{\lambda}{2}f\cdot B +\dot{f}\cdot \slashd B=0.
\end{align*} 
Case (2) follows as in the proof of Proposition \ref{adjtype3}, and the index theorem. 

For Case (3), we consider solutions of the form
\begin{align*}
\tilde{h}=f_0(t)\cdot\omega\odot dt+f_{1}(t)\cdot\mathcal{K}_{g_{Y}}(\omega),
\end{align*} 
which are \emph{not} in the image of $\mathcal{K}_{g}$, where $\omega\in\Lambda^{1}(T^{*}Y)$ is an eigenform of $\Delta_{H}$ with eigenvalue $\nu>0$ satisfying $\delta_{Y}\omega=0$. 
Case (3) then follows as in Proposition \ref{sec5prop}, and the index theorem. 

For Cases (4) and (5), we consider $\tilde{h}$ of the form
\begin{align*}
\tilde{h}=\mathcal{K}_{g}(\tilde{\omega}),
\end{align*} 
where $\tilde{\omega}\in\Lambda^{1}(T^{*}M)$.
The equation $\delta\tilde{h}=0$ says that $\omega$ 
is a solution of $\square_{\mathcal{K},g}\tilde{\omega}=0$, and
the solutions of this equations were completely classified in 
Section \ref{mixedsec} into those of 
types (a) and (b). 
Cases (4) and (5) then follow
from Proposition \ref{sec6prop} and the index theorem.
\end{proof}

\begin{proof}[Proof of Corollaries \ref{t1s3} and \ref{t1ac}]
Corollary \ref{t1s3} follows immediately from 
Theorem \ref{s3thm}. Corollary \ref{t1ac} then follows 
using a standard argument that solutions of elliptic equations in weighted spaces 
admit asymptotic expansions with leading terms solutions 
on the cylinder corresponding to indicial roots \cite{LockhartMcOwen}.
\end{proof}

\begin{proof}[Proof of Theorem \ref{t1a}]
Applying the divergence operator to the equation 
$\mathcal{D}^* Z = \mathcal{K}_g \omega$, we see that 
$\omega$ satisfies $\square_{\mathcal{K}} \omega = 0$.
An integration by parts shows that $\mathcal{K}_g \omega = 0$,
which implies that $\omega = 0$ since there are no nontrivial 
decaying conformal Killing fields. We next convert $(M,g)$ into 
a manifold with a cylindrical end, using the conformal 
factor $u^{-2}$ which is smooth and positive and equal to 
$r^{-2}$ outside of some compact set, and let $\hat{g} = r^{-2} g$. 
From conformal invariance of $\mathcal{D}^*$, we have that 
$\mathcal{D}_{\hat{g}}^* Z = 0$. Using Corollary \ref{t1ac}, 
we conclude that $ |Z|_{\hat{g}} = O(e^{-2t})$, where $t = \log(u)$
as $t \rightarrow \infty$. This implies that $|Z|_{g} = O(r^{-4})$
as $r \rightarrow \infty$. 
 
Next, if $h$ solves $\mathcal{D}h = 0$ and $\delta h = 0$, then 
$B'(h) = \mathcal{D}^* \mathcal{D} h = 0$, where $B'$ is 
is the linearized Bach tensor \cite{Itoh}. 
Since $B'$ is asymptotic to $\Delta^2$ as $r \rightarrow \infty$, 
\cite[Proposition 2.2]{AcheViaclovsky}, implies that there is no $O(r^{-1})$ term 
in the asymptotic expansion of $h$ and therefore 
$h = O(r^{-2})$ as $r \rightarrow \infty$.
\end{proof}

\begin{proof}[Proof of Theorem \ref{hypthm}]
The cokernel statements follow from combining Propositions \ref{sec5prop} 
and \ref{sec6prop} for the case $\kappa = -1$. 
The kernel statements follow from an analysis similar to the one outlined in the 
proof of Theorem \ref{kerlong}, using the index theorem. 
For the indicial root of $0$, the corresponding kernel of $F^*$ is 
of dimension $1 + b_1(Y) + 2 \dim(H^1_C(Y))$. 
From Theorem~\ref{mainprop}, we see that $3 dt\otimes dt - g_Y$  
is in the kernel of $F$. 
For a harmonic $1$-form $\omega$, from Theorem~\ref{mainprop} we also 
see that $\omega \odot dt$ is also in the kernel of $F$. 
For a traceless Codazzi tensor $B$ on $Y^3$,  from Theorem~\ref{mainprop}
it follows that 
\begin{align}
(c_1 \cos(t) + c_2 \sin(t) ) B,
\end{align}
is in the kernel of $F$ for any constants $c_1$ and $c_2$. 
By counting dimensions and using the index
theorem, this accounts for all kernel elements of $F$ corresponding 
to the indicial root $0$. 

\end{proof}

\begin{proof}[Proof of Theorem \ref{hrex}]
A compact hyperbolic $3$-manifold $(Y,g_Y)$ corresponds to a
discrete cocompact subgroup $\Gamma \subset O_{\circ}(3,1)$ without 
torsion. The space of locally conformally flat deformations of $Y$ is 
given by $H^1(\Gamma, \mathfrak{g})$,
where $\mathfrak{g}$ is the lie
algebra to $O(4,1)$ viewed as a $\Gamma$--module under the adjoint
representation.
If $Y^3$ is a hyperbolic rational homology $3$-sphere, then 
by assumption $b^1(Y) = 0$, so Theorem \ref{hypthm} implies 
that $H^2_+(\R \times Y^3) = \{0\}$ if and only if $H^1_C( Y) = \{0\}$.

In \cite{Kapovich} it 
was shown that infinitely many $(p,q)$-surgeries on a hyperbolic $2$-bridge knot 
satisfy $H^1(\Gamma, \mathfrak{g} )= \{0\}$ (a $2$-bridge knot is any knot 
that my be embedded in $\R^3$ with only $2$ local maxima, and the figure 
$8$ knot is an example of a hyperbolic $2$-bridge knot). These have $p \geq 2$ and are
therefore rational homology $3$-spheres, and all but finitely many 
are hyperbolic by Thurston's hyperbolic Dehn surgery Theorem (see, for example,
\cite{HodgsonKerckhoff} or \cite{PetronioPorti}).
By \cite[Lemma 6]{Lafontaine}, there is an 
injection $H^1_C(Y) \hookrightarrow H^1(\Gamma, \mathfrak{g})$,
so these examples are therefore an infinite family of hyperbolic 
rational homology $3$-spheres satisfying $H^2_+ (\R \times Y^3) = \{0\}$. 

Next, it was shown by DeBlois that there are infinitely many 
hyperbolic rational homology $3$-spheres containing closed 
embedded totally geodesic surfaces \cite{DeBlois} (these examples 
are $n$-fold cyclic branched covers of $S^3$ branched 
along a certain $2$-component link). By 
\cite[Theorem 2]{Lafontaine}, such a surface yields
a non-trivial traceless Codazzi tensor field on $Y$. Thus by 
Theorem \ref{hypthm}, the examples of DeBlois are an infinite family of 
examples of hyperbolic rational homology $3$-spheres 
satisfying $H^2_+(\R \times Y^3) \neq \{ 0 \}$. 

\end{proof}

\begin{proof}[Proof of Corollary \ref{hypthm2}]
This follows from Theorem \ref{hypthm}, again using a standard argument 
that solutions of elliptic equations in weighted spaces 
admit asymptotic expansions with leading terms solutions 
on the cylinder corresponding to indicial roots \cite{LockhartMcOwen}.
\end{proof}

\begin{proof}[Proof of Theorem \ref{flatthm}] 
The cokernel statements follow from combining Propositions \ref{sec5prop} 
and \ref{sec6prop} for the case $\kappa = 0$. 
The kernel statements follow from an analysis similar to the one outlined in the 
proof of Theorem \ref{kerlong}, using the index theorem. 
\end{proof}

%%%%%%%%%%%%%%%%%%%%%%%%%%%%%%%%%%%%%%%%
\section{The gluing problem}
\label{gluesec}

We will next describe the setup to the gluing theorem of 
Kovalev-Singer.  A brief statement of the theorem is 
as follows. 
\begin{theorem}[Floer, Kovalev-Singer , Donaldson-Friedman
\cite{Floer, KovalevSinger, DonaldsonFriedman}] Let
$(X_1,[g_1])$ and $(X_2, [g_2])$ be self-dual conformal 
structures on compact $4$-manifolds $X_i$ 
satisfying $H^2_c(X_i, [g_i]) = 0$ for $i = 1,2$. 
Then the connect sum $X_1 \# X_2$ admits self-dual conformal 
structures. 
\end{theorem}
Donaldson-Friedman proved this using twistor theory, using methods 
from the deformation theory of singular complex $3$-folds.  
The proofs of Floer and Kovalev-Singer are analytic,
and thus generalize more easily to the 
setting of orbifolds. Consequently, the gluing can be 
performed at isolated orbifold points $p_i \in X_i, i = 1,2$,  
provided they are compatible. This means that there is an 
orientation-reversing intertwining map between the actions
of the respective orbifold groups $\Gamma_i \subset SO(4)$ 
at the gluing points. 

We will next outline the idea of the analytic proof. 
Let $r_i(x) = d(p_i, x)$ in sufficiently 
small neighborhoods of $p_i$, and extend to smooth positive 
functions on each $X_i$. Consider the conformal cylindrification of 
$X_i$, which is $\tilde{X}_i = X_i \setminus \{p_i\}$ with metric 
$\tilde{g}_i = r_i^{-2} g_i$, and let $t_i = - \log r_i$. 
These metrics are then ``glued'' together with a 
cylindrical region in between using cutoff functions, we refer the reader to 
\cite[Section 2.3]{KovalevSinger} for the exact formulas. 
We only need to remark here that the main argument of \cite{KovalevSinger} 
is to reduce the gluing problem 
to the study of the deformation complex on the component cylindrified 
spaces. A weight $\delta > 0$ and a weight function are chosen so that in the limit, the 
weight function is $e^{\delta t}$ in the middle cylindrical 
region, and $e^{\delta t_1}$ on $\tilde{X}_1$ and 
$e^{- \delta t_2} $ on $\tilde{X}_2$.
One next considers the operators 
\begin{align}
F_1 &: e^{\delta t_1} C^{k, \alpha}(A_1)   
\xrightarrow{  \mathcal{D} \oplus 2 \delta_{\tilde{g}_1}}
e^{\delta t_1} C^{k-2, \alpha}(B_1)
\oplus e^{\delta t_1} C^{k-1, \alpha}(C_1), \\
F_0 &: e^{\delta t} C^{k, \alpha}(A_0)   
\xrightarrow{  \mathcal{D} \oplus 2 \delta_{\tilde{g}_0}}
e^{\delta t} C^{k-2, \alpha}(B_0)
\oplus e^{\delta t} C^{k-1, \alpha}(C_0),\\
F_2 &: e^{-\delta t_2} C^{k, \alpha}(A_2)   
\xrightarrow{  \mathcal{D} \oplus 2 \delta_{\tilde{g}_2}}
e^{-\delta t_2} C^{k-2, \alpha}(B_2)
\oplus e^{-\delta t_2} C^{k-1, \alpha}(C_2),
\end{align}
where $A_i = T^*\tilde{X}_i, B_i = S^2_0(T^*\tilde{X}_i)$, and 
$C_i = S^2_0(\Lambda^2_-)(T^*\tilde{X}_i)$. 
The adjoints of these operators are maps
\begin{align}
F_1^* &: e^{-\delta t_1} C^{k, \alpha}(B_1)
\oplus e^{-\delta t_1} C^{k-1, \alpha}(C_1) 
\rightarrow  e^{-\delta t_1} C^{k-2, \alpha}(A_1)   \\
F_0^* &: e^{-\delta t} C^{k, \alpha}(B_0)
\oplus e^{-\delta t} C^{k-1, \alpha}(C_0)
\rightarrow e^{-\delta t} C^{k-2, \alpha}(A_0)   \\
F_2^* &: e^{\delta t_2} C^{k, \alpha}(B_2)
\oplus e^{\delta t_2} C^{k-1, \alpha}(C_2)
\rightarrow
e^{\delta t_2} C^{k-2, \alpha}(A_2),
\end{align}
given by 
\begin{align}
F_i^* (Z, \omega) = \mathcal{D}^* Z - \mathcal{K}_{\tilde{g}_i} \omega. 
\end{align}
Note the duals of the H\"older spaces are not H\"older spaces, 
but we are only interested in the kernel and cokernel, which 
will be smooth by elliptic regularity, so this slight 
abuse of notation does not matter. 

On the middle cylindrical region, Corollary \ref{t1s3} 
shows that $F_0$ is an isomorphism for $0< \delta < 2$. 
On $\tilde{X_1}$,
we consider solutions of $F_1^*(Z,\omega) =0$
with both $Z = O(e^{-\delta t_1})$ and $\omega = O(e^{-\delta t_1})$ as 
$t_1 \rightarrow \infty$. 
Corollary \ref{t1ac} implies 
that $\omega$ is a conformal Killing field, and $Z =  O(e^{-2 t_1})$ as 
$t_1 \rightarrow \infty$. 
The conformal transformation formula 
$ \mathcal{D}^*_{\hat{g}} Z = r_1^2 \mathcal{D}^*_{g} Z$ shows 
that $Z$ is a solution of $\mathcal{D}^*_g Z = 0$ on $X_2 \setminus \{p_2\}$
satisfying $Z = O(1)$ as $r_1 \rightarrow 0$. The operator  
$\mathcal{D} \mathcal{D}^*$ is an elliptic operator with leading term 
$\Delta^2$ (\cite{AcheViaclovsky, Streets}). Consequently, 
the singularity is removable, so $Z$ extends to a smooth 
solution of $\mathcal{D}^* Z = 0$ on $X_1$, and $Z$ then 
vanishes by the assumption that $H^2_c(X_1, [g_1]) =0$.

On $\tilde{X_2}$, we consider solutions of $F_2^*(Z, \omega)= 0$ 
with both $Z = O(e^{\delta t_2})$ and $\omega = O(e^{\delta t_2})$ as 
$t_2 \rightarrow \infty$. 
Vanishing of $Z$ again follows from Corollary \ref{t1ac} and the assumption 
that $H^2_c(X_2, [g_2]) =0$.

\begin{remark}{\em
The argument given on \cite[page 1259-1260]{KovalevSinger} to handle 
the case of $\tilde{X_2}$ is incorrect, because there was a mistake in the order of 
growth given there. Namely, the growth rate given on the 
bottom on page 1258 for $H^{2, \pm}$ should be 
$|\Psi|_0 = O(r^{-2 \pm \delta})$, and not
$|\Psi|_0 = O(r^{-2 \mp \delta})$ as written there and 
then applied incorrectly in the subsequent argument.
Indeed, on $X_2$ the weight function 
is  $ e^{- \delta t}$, while the argument 
given there to remove the singularity (quoting Biquard's Theorem from \cite{Biquard}) 
requires $\delta > 0$.  
The above argument fixes this gap.}
\end{remark}
The remainder of the proof then proceeds as in \cite{KovalevSinger}.
\begin{remark} {\em
We note that there can be asymptotic cokernel arising from conformal 
Killing fields on the factors. Namely, on $\tilde{X}_1$ there 
are conformal 
Killing fields in the cokernel satisfying 
$\omega = O(e^{- \delta t_1})$ as $t_1 \rightarrow \infty$. 
The conformal transformation formula 
$\mathcal{K}_{\hat{g}}(\omega) = r_1^{-2} \mathcal{K}_g(r_1^2 \omega)$
shows that $r_1^2 \omega$ is 
a conformal Killing field on $X_1$ satisfying 
$ | r_1^2 \omega|_{g_1} = O(r_1^{1 + \delta})$ as $r_1 \rightarrow 0$. 
Thus the asymptotic cokernel contains conformal Killing fields on 
$X_1$ which vanish at $p_1$ and whose first derivatives vanish at $p$.   
Similarly, on $X_2$ there are conformal Killing fields in the 
cokernel satisfying $\omega = O(e^{ \delta t_2})$ as  $t_2 \rightarrow \infty$.
Then $r_2^2 \omega$ is a conformal Killing field on $X_2$ satisfying 
$ | r^2 \omega|_{g_2} = O(r_2^{1 - \delta})$ as $r_2 \rightarrow 0$. 
Thus the asymptotic cokernel also contains conformal Killing fields on 
$X_2$ which vanish at $p_2$.  However, the existence of this cokernel 
does not affect finding a self-dual metric, since we only need 
to find a zero of the first component of $F$ (and do not 
necessarily need to find a zero of the divergence map). 
}
\end{remark}

\appendix
%%%%%%%%%%%%%%%%%%%%%%%%%%%%%%%%%%%%%%%%%%%%%%%%%%%%%%%%%%%%%%%%%%%%%%%%%%%%%%%%%%%%
\section{The square of $\slashd$} 
\label{appendix}
In this appendix, we give the proof of Proposition \ref{squaredprop}. 
\begin{proof}[Proof of Proposition \ref{squaredprop}]
In a local orthonormal basis we have
\begin{align*}
\slashd^{2} h_{ij}&=\sum_{a,b}\epsilon_{iab}\nabla_{a}(\slashd h)_{bj}+\sum_{c,d}\epsilon_{jcd}\nabla_{c}(\slashd h)_{di}.
\end{align*}
Expanding the right hand side, we obtain 
\begin{align*}
\slashd^{2} h_{ij} &=\sum_{a,b}\sum_{k,l}\epsilon_{iab}\epsilon_{bkl}\nabla_{a}\nabla_{k}h_{lj}+\sum_{a,b}\sum_{m,n}\epsilon_{iab}\epsilon_{jmn}\nabla_{a}\nabla_{m}h_{nb}\\
&+\sum_{c,d}\sum_{p,q}\epsilon_{jcd}\epsilon_{dpq}\nabla_{c}\nabla_{p}h_{qi}+\sum_{c,d}\sum_{u,v}\epsilon_{jcd}\epsilon_{iuv}\nabla_{c}\nabla_{u}h_{vd}\\
&=I+II+III+IV.
\end{align*}
Note that $I+III$ is twice the symmetrization of $I$ and $II+IV$ is twice the symmetrization of $II$, so it will suffice to compute $I$ and $II$. A straightforward computation shows that if we let $a,b$ be indices such that $\{i,a,b\}=\{1,2,3\}$ then
\begin{align*}
I=\nabla_{a}\nabla_{i}h_{aj}-\nabla_{a}\nabla_{a}h_{ij}+\nabla_{b}\nabla_{i}h_{bj}-\nabla_{b}\nabla_{b}h_{ij}.
\end{align*}
Commuting covariant derivatives we have
\begin{align*}
\nabla_{a}\nabla_{i}h_{aj}&=\nabla_{i}\nabla_{a}h_{aj}-R^{p}_{aia}h_{pj}-R^{p}_{aij}h_{ap}\\
&=\nabla_{i}\nabla_{a}h_{aj}-\kappa \left(\delta_{a}^{p}(g_{Y})_{ia}-\delta_{i}^{p}(g_{Y})_{aa}\right)h_{pj}\\
&- \kappa \left(\delta_{a}^{p}(g_{Y})_{ij}-\delta_{i}^{p}(g_{Y})_{aj}\right)h_{ap}\\
&=\nabla_{i}\nabla_{a}h_{aj}
+ \kappa \big(h_{ij}-h_{aa}(g_{Y})_{ij}-(g_{Y})_{aj}h_{ai}\big).
\end{align*}
Similarly,
\begin{align*}
\nabla_{b}\nabla_{i}h_{bj}=\nabla_{i}\nabla_{b}h_{bj}+ \kappa \big( h_{ij}-h_{bb}(g_{Y})_{ij}-(g_{Y})_{bj}h_{bi} \big).
\end{align*}
It follows that
\begin{align*}
\nabla_{a}\nabla_{i}h_{aj}+\nabla_{b}\nabla_{i}h_{bj}&=\nabla_{i}\left(\delta_{Y}\right)_{j}-\nabla_{i}\nabla_{i}h_{ij}\\
&+2 \kappa h_{ij}+ \kappa \left(h_{aa}+h_{bb}+h_{ii}\right)(g_{Y})_{ij}\\
&-\kappa \left(h_{ii}(g_{Y})_{ij}+h_{ai}(g_{Y})_{aj}+h_{bi}(g_{Y})_{bj}\right)\\
&=\nabla_{i}\left(\delta_{Y}\right)_{j}-\nabla_{i}\nabla_{i}h_{ij}+3 \kappa \tf(h),
\end{align*}
and clearly
\begin{align*}
I=\nabla_{i}\left(\delta_{Y}h\right)_{j}-\Delta h_{ij}+3 \kappa \tf(h).
\end{align*}
We conclude that
\begin{align*}
I+III=\mathcal{L}_{g_Y}(\delta_{Y}h) -2\Delta h+6 \kappa \tf(h).
\end{align*}
For $II$ we consider two cases. If $i\ne j$ we let $l$ be a an index such that $\{i,j,l\}$ such that $\{i,j,l\}=\{1,2,3\}$, then it is easy to see that $II$ equals
\begin{align}\label{IIinej}
II&=-\nabla_{j}\nabla_{i}h_{ll}+\nabla_{j}\nabla_{l}h_{il}+\nabla_{l}\nabla_{i}h_{lj}-\nabla_{l}\nabla_{l}h_{ij}\\
&=A_{1}+A_{2}+A_{3}+A_{4}.
\end{align}
For the terms in \eqref{IIinej} we have 
\begin{align*}
A_{1}&=-\nabla_{j}\nabla_{i}h_{ll}=-\nabla_{j}\nabla_{i}\tr_{Y}(h)+\nabla_{j}\nabla_{i}h_{ii}+\nabla_{j}\nabla_{i}h_{jj},\\
A_{2}&=\nabla_{j}\nabla_{l}h_{il}=\nabla_{j}(\delta_{Y}h)_{i}-\nabla_{j}\nabla_{i}h_{ii}-\nabla_{j}\nabla_{j}h_{ij},
\end{align*}
For $A_{3}$ we commute covariant derivatives
\begin{align*}
A_{3}&=\nabla_{l}\nabla_{i}h_{lj}=\nabla_{i}\nabla_{l}h_{lj}-R^{p}_{lil}h_{pj}-R^{p}_{lij}h_{lp}\\
&=\nabla_{i}\nabla_{l}h_{lj}- \kappa \left(\delta_{l}^{p}g_{il}-\delta_{i}^{p}g_{ll}\right)h_{pj}- \kappa \left(\delta_{l}^{p}g_{ij}-\delta_{i}^{p}g_{lj}\right)h_{lp}\\
&=\nabla_{i}\nabla_{l}h_{lj}+ \kappa h_{ij}\\
&=\nabla_{i}\left(\delta_{Y} h\right)_{j}-\nabla_{i}\nabla_{i}h_{ij}-\nabla_{i}\nabla_{j}h_{ij}
+ \kappa h_{ij},
\end{align*}
and finally
\begin{align*}
A_{4}=-\nabla_{l}\nabla_{l}h_{ij}=-\Delta_{g_{Y}}h_{ij}+\nabla_{i}\nabla_{i}h_{ij}+\nabla_{j}\nabla_{j}h_{ij}.
\end{align*}
We then have
\begin{align*}
A_{1}+A_{2}=-\nabla_{j}\nabla_{i}\tr_{Y}(h)+\nabla_{j}(\delta_{Y}h)_{i}+\nabla_{j}\nabla_{i}h_{jj}-\nabla_{j}\nabla_{j}h_{ij},
\end{align*}
and 
\begin{align*}
A_{3}+A_{4}&=\nabla_{i}\left(\delta_{Y} h\right)_{j}-\nabla_{i}\nabla_{i}h_{ij}-\nabla_{i}\nabla_{j}h_{ij} + \kappa h_{ij}\\
&-\Delta_{g_{Y}}h_{ij} +\nabla_{i}\nabla_{i}h_{ij}+\nabla_{j}\nabla_{j}h_{ij}\\
&=\nabla_{i}\left(\delta_{Y} h\right)_{j} + \kappa h_{ij}-\Delta_{g_{Y}}h_{ij}\\
&-\nabla_{i}\nabla_{j}h_{jj}+\nabla_{j}\nabla_{j}h_{ij},
\end{align*}
so then
\begin{align*}
II&=A_{1}+A_{2}+A_{3}+A_{4}\\
&=-\nabla_{j}\nabla_{i}\tr_{Y}(h)-\Delta_{g_{Y}} h_{ij}\\
&+\mathcal{L}_{g_Y}(\delta_{Y}h)
+ \kappa h_{ij}+\nabla_{j}\nabla_{i}h_{jj}-\nabla_{i}\nabla_{j}h_{jj}.
\end{align*}
Commuting covariant derivatives we have
\begin{align*}
\nabla_{j}\nabla_{i}h_{jj}-\nabla_{i}\nabla_{j}h_{jj}&=-R^{p}_{jij}h_{pj}-R^{p}_{jij}h_{jp}=-2R^{p}_{jij}h_{pj}\\
&=-2 \kappa \left(\delta_{j}^{p}g_{ij}-\delta_{i}^{p}g_{jj}\right)h_{jp}\\
&=2 \kappa h_{ij},
\end{align*}
so we have shown
\begin{align*}
II=-\Delta h_{ij}-\nabla_{i}\nabla_{j}\tr_{Y}(h)+\mathcal{L}_{g_Y}(\delta_{Y}h)
+3 \kappa h_{ii},
\end{align*}
so then $II+IV$ is given by
\begin{align}\label{IIinejfinal}
II+IV=-2\Delta_{g_{Y}}h-2\nabla_{i}\nabla_{j}\tr_{Y}(h)+2
\mathcal{L}_{g_Y}(\delta_{Y}h) +6 \kappa h_{ij}.
\end{align}
If now $i=j$ we let $a,b$ be indices such that $\{i,a,b\}=\{1,2,3\}$ so that we have
\begin{align*}
II=\nabla_{a}\nabla_{a}h_{bb}-\nabla_{a}\nabla_{b}h_{ab}-\nabla_{b}\nabla_{a}h_{ab}+\nabla_{b}\nabla_{b}h_{aa},
\end{align*}
which simplifies to
\begin{align*}
II&=\Delta_{g_{Y}}\tr_{Y}(h)-\nabla_{i}\nabla_{i}\tr_{Y}(h)-\left(\delta_{Y}\delta_{Y} h\right)+\nabla_{i}\left(\delta h\right)_{i}-\Delta_{g_{Y}} h_{ii}\\
&+\nabla_{i}\nabla_{i}h_{ii}+\nabla_{a}\nabla_{i}h_{ai}+\nabla_{b}\nabla_{i}h_{ib}.
\end{align*}
Commuting covariant derivatives we obtain
\begin{align*}
II&=\Delta_{g_{Y}}\tr_{Y}(h)-\nabla_{i}\nabla_{i}\tr_{Y}(h)-\left(\delta_{Y}\delta_{Y} h\right)-\Delta_{g_{Y}} h_{ii}\\
&+\nabla_{i}(\delta_{Y} h)_{i}-R^{p}_{aia}h_{pi}-R^{p}_{aii}h_{ap}-R^{p}_{bii}h_{pb}-R^{p}_{bib}h_{ip},
\end{align*}
and it is easy to see from this expression that
\begin{align}\label{IIieqj}
II=\Delta_{g_{Y}}\tr_{Y}(h)-\nabla_{i}\nabla_{i}\tr_{Y}(h)-\left(\delta_{Y}\delta_{Y} h\right)-\Delta_{g_{Y}}h_{ii}+2\nabla_{i}\left(\delta_{Y}h\right)_{i}+3 \kappa \tf(h),
\end{align}
so then 
\begin{align*}
II+IV&=2\Delta_{g_{Y}}\tr_{Y}(h)-2\nabla_{i}\nabla_{i}\tr_{Y}(h)-2\left(\delta_{Y}\delta_{Y} h\right)\\
&-2\Delta_{g_{Y}}h_{ii}+4\nabla_{i}\left(\delta_{Y}h\right)_{i}+6 \kappa \tf(h).
\end{align*}
Now, since we are using a local orthonormal basis to compute $\slashd^{2} h$,  it is clear that  in either case $i=j$ or $i\ne j$  the results in \eqref{IIinejfinal} and \eqref{IIieqj} are equivalent to
\begin{align*}
II+IV&=-2\Delta_{g_{Y}}\tf(h)-2{\stackrel{\circ}{\nabla^{2}}}h
+2\mathcal{L}_{g_Y}(\delta_{Y}h)
-2\left(\delta_{Y}\delta_{Y}\right)g_{Y}\\
&+\frac{2}{3}\left(\Delta_{g_{Y}}\tr_{Y}(h)\right)g_{Y}+6 \kappa \tf(h).
\end{align*}
Expressing $I+III$ as 
\begin{align*}
I+III=-2\Delta_{g_{Y}}\tf(h)-\frac{2}{3}\Delta_{g_{Y}}\tr_{Y}(h)g_{Y}
+\mathcal{L}_{g_Y}(\delta_{Y}h)
+6 \kappa \tf(h),
\end{align*}
we conclude that
\begin{align*}
\slashd^{2}h=-4\Delta_{g_{Y}}\tf(h)-2\stackrel{\circ}{\nabla^{2}}\tr_{Y}(h)+3\mathcal{K}_{g_{Y}}(\delta_{Y}h)+12 \kappa \tf(h),
\end{align*}
as needed.
\end{proof}
The following corollary should be compared with \cite[Lemma 5.1]{Floer}:
\begin{corollary}
\label{Floercor}
For any $h\in S^{2}(T^{*}Y)$ we have
\begin{align}
 E^{\prime}_{g}(h)=\frac{1}{8}\slashd^{2}h-\frac{1}{4}\stl{\circ}{\nabla^{2}}\tr_{Y}(h)+\frac{1}{8}\mathcal{K}_{g_{Y}}(\delta_{Y} h)-\frac{\kappa}{2}\tf(h).
\end{align}
\end{corollary}
 \bibliography{Adjoint_ASD}

\def\cprime{$'$}
\providecommand{\bysame}{\leavevmode\hbox to3em{\hrulefill}\thinspace}
\providecommand{\MR}{\relax\ifhmode\unskip\space\fi MR }
\providecommand{\MRhref}[2]{%
  \href{http://www.ams.org/mathscinet-getitem?mr=#1}{#2}
}
\providecommand{\href}[2]{#2}
\begin{thebibliography}{DeB06}

\bibitem[AVis]{AcheViaclovsky}
Antonio~G. Ache and Jeff~A. Viaclovsky, \emph{Obstruction flat asymptotically
  locally {E}uclidean metrics}, arXiv.org:1106.1249, 2011, to appear in
  Geometric and Functional Analysis.

\bibitem[Biq91]{Biquard}
Olivier Biquard, \emph{Fibr\'es paraboliques stables et connexions
  singuli\`eres plates}, Bull. Soc. Math. France \textbf{119} (1991), no.~2,
  231--257.

\bibitem[Che09]{Chen}
Szu-Yu~Sophie Chen, \emph{Optimal curvature decays on asymptotically locally
  euclidean manifolds}, arXiv.org:0911.5538, 2009.

\bibitem[DeB06]{DeBlois}
Jason DeBlois, \emph{Totally geodesic surfaces and homology}, Algebr. Geom.
  Topol. \textbf{6} (2006), 1413--1428 (electronic).

\bibitem[DF89]{DonaldsonFriedman}
S.~Donaldson and R.~Friedman, \emph{Connected sums of self-dual manifolds and
  deformations of singular spaces}, Nonlinearity \textbf{2} (1989), no.~2,
  197--239.

\bibitem[DN55]{DN}
Avron Douglis and Louis Nirenberg, \emph{Interior estimates for elliptic
  systems of partial differential equations}, Comm. Pure Appl. Math. \textbf{8}
  (1955), 503--538.

\bibitem[Flo91]{Floer}
Andreas Floer, \emph{Self-dual conformal structures on {$l{\bf C}{\rm P}^2$}},
  J. Differential Geom. \textbf{33} (1991), no.~2, 551--573.

\bibitem[Fol89]{foll}
G.~B. Folland, \emph{Harmonic analysis of the de {R}ham complex on the sphere},
  J. Reine Angew. Math. \textbf{398} (1989), 130--143.

\bibitem[HK05]{HodgsonKerckhoff}
Craig~D. Hodgson and Steven~P. Kerckhoff, \emph{Universal bounds for hyperbolic
  {D}ehn surgery}, Ann. of Math. (2) \textbf{162} (2005), no.~1, 367--421.

\bibitem[Ito95]{Itoh}
Mitsuhiro Itoh, \emph{The {W}eitzenb\"ock formula for the {B}ach operator},
  Nagoya Math. J. \textbf{137} (1995), 149--181.

\bibitem[Kap94]{Kapovich}
Michael Kapovich, \emph{Deformations of representations of discrete subgroups
  of {${\rm SO}(3,1)$}}, Math. Ann. \textbf{299} (1994), no.~2, 341--354.

\bibitem[Koi78]{Koiso1}
Norihito Koiso, \emph{Nondeformability of {E}instein metrics}, Osaka J. Math.
  \textbf{15} (1978), no.~2, 419--433.

\bibitem[KS01]{KovalevSinger}
A.~Kovalev and M.~Singer, \emph{Gluing theorems for complete anti-self-dual
  spaces}, Geom. Funct. Anal. \textbf{11} (2001), no.~6, 1229--1281.

\bibitem[Laf83]{Lafontaine}
Jacques Lafontaine, \emph{Modules de structures conformes plates et cohomologie
  de groupes discrets}, C. R. Acad. Sci. Paris S\'er. I Math. \textbf{297}
  (1983), no.~13, 655--658.

\bibitem[LM85]{LockhartMcOwen}
Robert~B. Lockhart and Robert~C. McOwen, \emph{Elliptic differential operators
  on noncompact manifolds}, Ann. Scuola Norm. Sup. Pisa Cl. Sci. (4)
  \textbf{12} (1985), no.~3, 409--447.

\bibitem[LM08]{LeBrunMaskit}
Claude LeBrun and Bernard Maskit, \emph{On optimal 4-dimensional metrics}, J.
  Geom. Anal. \textbf{18} (2008), no.~2, 537--564.

\bibitem[Poo86]{Poon}
Y.~Sun Poon, \emph{Compact self-dual manifolds with positive scalar curvature},
  J. Differential Geom. \textbf{24} (1986), no.~1, 97--132.

\bibitem[PP00]{PetronioPorti}
Carlo Petronio and Joan Porti, \emph{Negatively oriented ideal triangulations
  and a proof of {T}hurston's hyperbolic {D}ehn filling theorem}, Expo. Math.
  \textbf{18} (2000), no.~1, 1--35.

\bibitem[Str10]{Streets}
Jeffrey Streets, \emph{Asymptotic curvature decay and removal of singularities
  of {B}ach-flat metrics}, Trans. Amer. Math. Soc. \textbf{362} (2010), no.~3,
  1301--1324.

\bibitem[TV05]{TV}
Gang Tian and Jeff Viaclovsky, \emph{Bach-flat asymptotically locally
  {E}uclidean metrics}, Invent. Math. \textbf{160} (2005), no.~2, 357--415.

\end{thebibliography}
\end{document}